\documentclass{amsart}

\makeatletter
\@namedef{subjclassname@2020}{%
  \textup{2020} Mathematics Subject Classification}
\makeatother

\usepackage[T1]{fontenc}
\usepackage{amsmath,amssymb,amsfonts,mathtools}
\usepackage{amsthm}
\usepackage[dvipsnames]{xcolor}
\usepackage{enumitem}
\usepackage{microtype}

\newtheorem{theorem}{Theorem}[section]
\newtheorem{lemma}[theorem]{Lemma}
\newtheorem{proposition}[theorem]{Proposition}
\newtheorem{corollary}[theorem]{Corollary}

\theoremstyle{definition}
\newtheorem{definition}[theorem]{Definition}
\newtheorem{example}[theorem]{Example}

\theoremstyle{remark}
\newtheorem{remark}[theorem]{Remark}

\numberwithin{equation}{section}

\usepackage{hyperref}
\hypersetup{
    colorlinks=true,
    linkcolor=Blue,
    citecolor=ForestGreen,
    filecolor=Blue,
    urlcolor=BrickRed,
    linktocpage=true
}

\begin{document}

\title[An Algebraic View of Amenable Groupoids]{An Algebraic View of Amenable Groupoids}

\author{Karol Herrera}
\address{Universidad ECCI}
\curraddr{Campus Universitario}
\email{kherrerac@ecci.edu.co}

\author{Andr\'es Rubiano}
\address{Universidad Distrital Francisco Jos\'e de Caldas}
\curraddr{Campus Universitario}
\email{aarubianos@udistrital.edu.co}



\date{}
\dedicatory{Dedicated to Juan Hern\'andez}

\begin{abstract}
We introduce source-anchored algebraic amenability for Hausdorff ample groupoids, a F{\o}lner condition formulated in the Steinberg algebra and localized at compact open subsets of the unit space. The unanchored bisectional condition is exactly algebraic amenability of the Steinberg algebra, whereas the anchored version recovers ordinary amenability for one-unit groupoids and detects behavior invisible to the unanchored theory. We establish permanence under isomorphisms, open invariant reductions, disjoint unions, finite products, directed unions and AF approximations, while arbitrary subgroupoids need not inherit the property. For countable transitive discrete groupoids, infinite unit spaces are amenable through range escape, whereas finite unit spaces are amenable precisely when their isotropy groups are. This yields a characterization of hereditary algebraic amenability in terms of isotropy.
\end{abstract}

\maketitle

\section{Introduction}\label{introduction}

Amenability provides a way of replacing exact invariance by finite configurations whose boundary becomes negligible. For a discrete group, invariant means, F{\o}lner sets, fixed-point properties and the absence of paradoxical decompositions all describe the same phenomenon; see \cite{vNeu29,Fol55,Day57,Gre69,Pat88}. Among these formulations, the F{\o}lner condition is especially suitable for algebraic purposes, since it measures invariance through the growth of finite sets under a prescribed finite family of translations.

Groupoids retain the local symmetry of groups while allowing multiplication only along compatible source and range fibers. Their amenability theory has been developed mainly through analytic and dynamical methods, involving Haar systems, invariant means, representations and groupoid $C^*$-algebras; see \cite{Ren80,ADR00}. Recent work continues this program through exactness and weak containment \cite{AnantharamanDelaroche23}, weak amenability and completely bounded approximation properties \cite{Pacheco25}, and amenability of continuous unitary representations on Hilbert bundles \cite{SridharanKumar26}. The purpose of the present paper is different: we seek a genuinely algebraic F{\o}lner theory adapted to ample groupoids.

For a Hausdorff ample groupoid $\mathcal G$, the Steinberg algebra $A_{\Bbbk}(\mathcal G)$ records the compact open geometry of $\mathcal G$. It is generated by the characteristic functions of compact open bisections, and convolution mirrors bisection multiplication through
$$
\mathbf 1_U\ast\mathbf 1_V=\mathbf 1_{UV}.
$$
Steinberg algebras include group algebras, inverse semigroup algebras and Leavitt path algebras among their basic examples; see \cite{Ste10,CFST14,ClarkHazrat20}. When $\mathcal G$ is a discrete group $G$ viewed as a one-unit groupoid, one simply recovers
$$
A_{\Bbbk}(\mathcal G)\cong\Bbbk G.
$$
This makes the algebraic F{\o}lner condition of Elek, Bartholdi, and Ara--Li--Lled\'o--Wu a natural starting point; see \cite{Elek03,Bartholdi08,AraLiLledoWu18}. In that setting, finite sets are replaced by non-zero finite-dimensional subspaces, and the boundary is measured by the increase in dimension under multiplication.

Applying algebraic amenability directly to $A_{\Bbbk}(\mathcal G)$ does not, however, fully reflect the unit-space geometry. A F{\o}lner subspace may be supported in one favorable component and avoid another component where non-amenable behavior occurs. This possibility is absent for groups, whose unit space consists of a single point. To retain the local information carried by $\mathcal G^{(0)}$, we require the approximating subspace to lie, for every non-empty compact open subset $C\subseteq\mathcal G^{(0)}$, in the principal left ideal
$$
A_{\Bbbk}(\mathcal G)\mathbf 1_C = \{f\in A_{\Bbbk}(\mathcal G)\mid \operatorname{supp}(f)\subseteq s^{-1}(C)\}.
$$
We call the resulting property \emph{source-anchored left algebraic amenability}. The anchor fixes the source of the approximation while leaving its range free, a feature that becomes decisive in the discrete theory.

The unanchored bisectional condition turns out to be exactly algebraic amenability of the Steinberg algebra. The anchored condition agrees with ordinary group amenability in the one-unit case and behaves well under the natural constructions available for ample groupoids: isomorphisms, open invariant reductions, disjoint unions, finite products, directed unions and finite-dimensional approximations. In particular, AF groupoids satisfy the condition. The subgroupoid problem is subtler. Open invariant reductions inherit amenability, but arbitrary Hausdorff ample subgroupoids do not; explicit discrete examples show that even an open subgroupoid may carry a non-amenable isotropy group inside an amenable ambient groupoid.

This failure is explained by a sharp distinction between finite and infinite orbits. A countable transitive discrete groupoid with infinitely many units always admits source-anchored F{\o}lner subspaces: after fixing the source, one can choose a range fiber outside the finite region detected by the prescribed multipliers. We refer to this mechanism as \emph{range escape}. When the unit space is finite, no such escape is possible. The Steinberg algebra is then a finite matrix algebra over an isotropy group algebra, and the F{\o}lner inequality descends to the isotropy group. Consequently, transitive discrete groupoids with finitely many units are algebraically amenable precisely when their isotropy groups are amenable. Combining the finite and infinite orbit cases yields a complete characterization of hereditary source-anchored algebraic amenability for countable discrete groupoids in terms of amenability of all isotropy groups.

The point of the theory is not to replace topological or measurewise amenability, nor to claim an unrestricted equivalence with them. Rather, it isolates a finite-dimensional approximation property intrinsic to Steinberg algebras and shows how the source and range geometry of a groupoid alters the familiar algebraic picture. The examples distinguish source anchoring from plain algebraic amenability of $A_{\Bbbk}(\mathcal G)$, without establishing a general implication between the anchored condition and the analytic notions of groupoid amenability.

Section~\ref{preliminaries} recalls the group-theoretic and groupoid background used throughout the paper. The algebraic F{\o}lner conditions are introduced in Section~\ref{algebraic-folner}, where the one-unit case is related to classical group amenability. Section~\ref{permanence} studies permanence, the failure of inheritance by arbitrary subgroupoids, and the role of isotropy in the discrete setting. Section~\ref{finite-index-transformation-groupoids} characterizes algebraic amenability for compact global actions and develops finite-index induction. The examples in Section~\ref{examples} give explicit approximations and distinguish the weak, source-anchored and hereditary conditions. Section~\ref{future-work} concludes with questions beyond the classes considered here.

\section{Preliminaries}\label{preliminaries}

Throughout this paper, $G$ denotes a countable discrete group with identity element $e_G$. If $X$ is a set, we write $\mathcal P(X)$ for its power set, that is,
$$
\mathcal P(X)=\{A\mid A\subseteq X\}.
$$
For $A\subseteq G$ and $g\in G$, we set
$$
gA=\{ga\mid a\in A\}, \ Ag=\{ag\mid a\in A\}.
$$
Disjoint unions will be denoted by $\sqcup$. The purpose of this section is to recall the classical notions of amenability for discrete groups and to introduce the basic language of ample groupoids and Steinberg algebras. The group case will serve as the one-object model for the groupoid constructions developed later.

\subsection{Amenable groups and finitely additive measures}

\begin{definition}[{\cite[\S 2.1]{Cortez22}}]
Let $X$ be a non-empty set. A \emph{finitely additive probability measure} on $X$ is a map
$$
\mu:\mathcal P(X)\longrightarrow [0,1]
$$
such that the following conditions hold:
\begin{enumerate}
    \item $\mu(X)=1$;
    \item if $A,B\subseteq X$ and $A\cap B=\varnothing$, then
    $$
    \mu(A\sqcup B)=\mu(A)+\mu(B).
    $$
\end{enumerate}
We denote by $\operatorname{PM}(X)$ the set of all finitely additive probability measures on $X$.
\end{definition}

\begin{definition}[{\cite[Definition 3]{Cortez22}; cf. \cite{vNeu29}}]
Let $G$ be a group. A measure $\mu\in\operatorname{PM}(G)$ is called \emph{left $G$-invariant} if
$$
\mu(gA)=\mu(A) \text{ for every }g\in G\text{ and every }A\subseteq G.
$$
Similarly, $\mu$ is called \emph{right $G$-invariant} if
$$
\mu(Ag)=\mu(A) \text{ for every }g\in G\text{ and every }A\subseteq G.
$$
If both conditions hold, then $\mu$ is called \emph{bi-invariant}.
\end{definition}

\begin{definition}[{\cite[Definition 3]{Cortez22}; cf. \cite{vNeu29,Day57}}]
A discrete group $G$ is called \emph{amenable} if there exists a left $G$-invariant finitely additive probability measure on $G$.
\end{definition}

This is one of the classical formulations of amenability. For the historical development of finitely additive invariant measures, see \cite{Ban23,vNeu29,Day57}. It expresses the existence of a translation-invariant averaging procedure on all subsets of the group, but only in the finitely additive sense. Thus, for a countably infinite discrete group, amenability replaces the impossible existence of a countably additive translation-invariant probability measure on $G$ by the existence of a finitely additive invariant probability measure.

\begin{example}[{\cite[Proposition 4]{Cortez22}}]
Every finite group is amenable. Indeed, if $G$ is finite, then
$$
\mu(A)=\frac{|A|}{|G|} \text{ for every }A\subseteq G
$$
defines an element of $\operatorname{PM}(G)$. Since left multiplication by any $g\in G$ is a bijection, one has $|gA|=|A|$, and therefore $\mu(gA)=\mu(A)$.
\end{example}

\subsection{Means on bounded functions}

Let $X$ be a non-empty set. We denote by $\ell^\infty(X)$ the Banach space of all bounded functions $f:X\longrightarrow\mathbb R$, endowed with the norm
$$
\|f\|_\infty=\sup_{x\in X}|f(x)|.
$$
For $A\subseteq X$, we denote by $\mathbf 1_A$ the characteristic function of $A$.

\begin{definition}[{\cite[Definition 15]{Cortez22}; cf. \cite{Fol55,Day57}}]
A \emph{mean} on $X$ is a linear functional
$$
m:\ell^\infty(X)\longrightarrow\mathbb R
$$
such that:
\begin{enumerate}
    \item $m(\mathbf 1_X)=1$;
    \item if $f\in\ell^\infty(X)$ and $f\geq 0$, then $m(f)\geq 0$.
\end{enumerate}
We denote by $\mathcal M(X)$ the set of all means on $X$.
\end{definition}

Every mean has norm one. Moreover, $\mathcal M(X)$ is a convex weak$^*$ compact subset of $\ell^\infty(X)^*$; see \cite[Proposition 16]{Cortez22}. This compact convex structure is one of the standard functional-analytic tools behind fixed-point characterizations of amenability.

For $g\in G$ and $f\in\ell^\infty(G)$, define the left translate of $f$ by
$$
(g\cdot f)(h)=f(g^{-1}h), \ h\in G.
$$
Similarly, define the right translate of $f$ by
$$
(f\cdot g)(h)=f(hg^{-1}), \ h\in G.
$$

\begin{definition}[{\cite[Definition 18]{Cortez22}; cf. \cite{Fol55,Day57}}]
A mean $m\in\mathcal M(G)$ is called \emph{left $G$-invariant} if
$$
m(g\cdot f)=m(f) \text{ for every }g\in G\text{ and every }f\in\ell^\infty(G).
$$
Right invariant and bi-invariant means are defined analogously.
\end{definition}

\begin{proposition}[{\cite[Lemma 17 and Lemma 19]{Cortez22}}]
There is an affine bijection
$$
\mathcal M(X)\longrightarrow\operatorname{PM}(X), \ m\longmapsto\mu_m,
$$
where
$$
\mu_m(A)=m(\mathbf 1_A) \text{ for every }A\subseteq X.
$$
Under this bijection, left $G$-invariant means on $G$ correspond exactly to left $G$-invariant finitely additive probability measures on $G$.
\end{proposition}

\begin{proposition}[{\cite[Theorem 20]{Cortez22}; cf. \cite{Day57,Gre69,Pat88}}]
For a discrete group $G$, the following conditions are equivalent:
\begin{enumerate}
    \item $G$ is amenable;
    \item there exists a left $G$-invariant mean on $\ell^\infty(G)$;
    \item there exists a right $G$-invariant mean on $\ell^\infty(G)$;
    \item there exists a bi-invariant mean on $\ell^\infty(G)$;
    \item there exists a left $G$-invariant finitely additive probability measure on $G$;
    \item there exists a bi-invariant finitely additive probability measure on $G$.
\end{enumerate}
\end{proposition}

\subsection{F{\o}lner sets}

Amenability also admits a finitary combinatorial formulation in terms of almost invariant finite subsets. This viewpoint is especially useful for the algebraic approach developed later, since it allows invariant means to be approximated by finite averaging procedures.

\begin{definition}[{\cite[p. 245, Main Theorem]{Fol55}; cf. \cite[\S 4]{Cortez22}}]
A countable discrete group $G$ satisfies the \emph{left F{\o}lner condition} if, for every finite subset $K\subseteq G$ and every $\varepsilon>0$, there exists a non-empty finite subset $F\subseteq G$ such that
$$
\frac{|F\setminus kF|}{|F|}<\varepsilon \text{ for every }k\in K.
$$
\end{definition}

Equivalently, $G$ admits a sequence $(F_n)_{n\geq 1}$ of non-empty finite subsets of $G$ such that
$$
\lim_{n\to\infty}\frac{|F_n\triangle gF_n|}{|F_n|}=0 \text{ for every }g\in G,
$$
where $\triangle$ denotes symmetric difference. Passing to inverses shows that the left and right versions of the F{\o}lner condition are equivalent.

\begin{proposition}[F{\o}lner criterion, {\cite[p. 245, Main Theorem]{Fol55}; cf. \cite[Proposition 33 and Theorem 36]{Cortez22}}]
For a countable discrete group $G$, the following conditions are equivalent:
\begin{enumerate}
    \item $G$ is amenable;
    \item $G$ satisfies the left F{\o}lner condition;
    \item $G$ satisfies the right F{\o}lner condition;
    \item $G$ admits a left F{\o}lner sequence;
    \item $G$ admits a right F{\o}lner sequence.
\end{enumerate}
\end{proposition}

If $(F_n)_{n\geq 1}$ is a left F{\o}lner sequence, then each $F_n$ defines a mean
$$
m_n(f)=\frac{1}{|F_n|}\sum_{h\in F_n}f(h), \ f\in\ell^\infty(G).
$$
The F{\o}lner condition implies that $(m_n)_{n\geq 1}$ is asymptotically invariant. By weak$^*$ compactness of $\mathcal M(G)$, the sequence admits a convergent subnet, and the limit of every such subnet is a left $G$-invariant mean.

\begin{example}[{\cite[Example 35 and Corollary 40]{Cortez22}}]
The group $\mathbb Z^d$ is amenable. Indeed, the sequence
$$
F_n=\{0,1,\dots,n-1\}^d, \ n\geq 1,
$$
is a F{\o}lner sequence. Consequently, every finitely generated abelian group is amenable, and permanence under directed unions implies that every abelian group is amenable.
\end{example}

\subsection{Paradoxical decompositions}

Amenability can also be characterized by the impossibility of duplicating a group by finitely many translations. This is the content of the Tarski--F{\o}lner theorem.

\begin{definition}[{\cite[\S 2.2]{Cortez22}; cf. \cite{BanTar24,Tar38}}]
A group $G$ is said to be \emph{paradoxical} if there exist pairwise disjoint subsets
$$
A_1,\dots,A_n,B_1,\dots,B_m\subseteq G
$$
and elements
$$
g_1,\dots,g_n,h_1,\dots,h_m\in G
$$
such that
$$
G=\bigsqcup_{i=1}^n g_iA_i=\bigsqcup_{j=1}^m h_jB_j.
$$
In this case, we say that $G$ admits a \emph{paradoxical decomposition}.
\end{definition}

\begin{example}[{\cite[Theorem 5]{Cortez22}}]
The free group $F_2=\langle a,b\rangle$ is not amenable. Indeed, by decomposing reduced words according to their first letter, one obtains a paradoxical decomposition of $F_2$. This is the classical obstruction behind the non-amenability of free groups of rank at least two.
\end{example}

\begin{theorem}[Tarski--F{\o}lner theorem, {\cite{Tar38,Fol55}; cf. \cite[Theorem 36]{Cortez22}}]
For a countable discrete group $G$, the following conditions are equivalent:
\begin{enumerate}
    \item $G$ is amenable;
    \item $G$ satisfies the F{\o}lner condition;
    \item $G$ does not admit a paradoxical decomposition.
\end{enumerate}
\end{theorem}

This theorem provides a fundamental bridge between invariant means, finite approximations and algebraic decompositions.

\subsection{Permanence properties}

Amenability is stable under several natural group-theoretic constructions.

\begin{proposition}[{\cite[Propositions 7, 13, 21 and 26 and Corollaries 22 and 27]{Cortez22}; cf. \cite{Pat88}}]
Let $G$ be a discrete group. The following assertions hold:
\begin{enumerate}
    \item if $G$ is amenable and $H\leq G$, then $H$ is amenable;
    \item if $G$ is amenable and $N\trianglelefteq G$, then $G/N$ is amenable;
    \item if $N\trianglelefteq G$ and both $N$ and $G/N$ are amenable, then $G$ is amenable;
    \item finite direct products of amenable groups are amenable;
    \item directed unions of amenable groups are amenable;
    \item if $G$ is countable, then $G$ is amenable if and only if every finitely generated subgroup of $G$ is amenable.
\end{enumerate}
\end{proposition}

\begin{corollary}
Finite groups, abelian groups, locally finite groups and solvable groups are amenable.
\end{corollary}

\begin{proof}
Finite groups are amenable by the normalized counting measure. Every finitely generated abelian group is isomorphic to $\mathbb Z^d\times F$ for some $d\geq 0$ and some finite abelian group $F$, and is therefore amenable. Permanence under directed unions then shows that every abelian group is amenable. Every finitely generated subgroup of a locally finite group is finite, so local finiteness implies amenability. Finally, solvable groups are obtained from abelian groups by finitely many extensions, and amenability is preserved under extensions.
\end{proof}

\subsection{Dynamical characterization}

\begin{definition}[{\cite[\S 1.1]{Cortez22}; cf. \cite{Gre69}}]
Let $G$ be a group and let $X$ be a non-empty set. A \emph{left action} of $G$ on $X$ is a map
$$
G\times X\longrightarrow X, \ (g,x)\longmapsto gx,
$$
such that
$$
e_Gx=x \text{ and } g(hx)=(gh)x
$$
for every $g,h\in G$ and every $x\in X$. If $X$ is a topological space and every map $x\longmapsto gx$ is a homeomorphism, then the action is called a \emph{topological action}. Since $G$ is discrete, this is equivalent to continuity of the action map $G\times X\longrightarrow X$.
\end{definition}

\begin{definition}[{\cite{Bogachev07}; cf. \cite[\S 1.1]{Cortez22}}]
Let $X$ be a topological space. The \emph{Borel $\sigma$-algebra} of $X$, denoted by $\mathcal B(X)$, is the smallest $\sigma$-algebra of subsets of $X$ containing every open subset of $X$. The elements of $\mathcal B(X)$ are called the \emph{Borel subsets} of $X$.

A \emph{Borel probability measure} on $X$ is a map
$$
\nu:\mathcal B(X)\longrightarrow [0,1]
$$
such that:
\begin{enumerate}
    \item $\nu(X)=1$;
    \item for every sequence $(A_n)_{n\geq 1}$ of pairwise disjoint Borel subsets of $X$, one has
    $$
    \nu\left(\bigsqcup_{n=1}^{\infty}A_n\right)=\sum_{n=1}^{\infty}\nu(A_n).
    $$
\end{enumerate}
We denote by $\operatorname{Prob}(X)$ the set of all Borel probability measures on $X$.
\end{definition}

\begin{definition}[{\cite[\S 1.1]{Cortez22};
cf. \cite{Gre69}}]
Let $G\curvearrowright X$ be a topological action on a compact Hausdorff space. A Borel probability measure $\nu\in\operatorname{Prob}(X)$ is called \emph{$G$-invariant} if
$$
\nu(gA)=\nu(A)
$$
for every $g\in G$ and every $A\in\mathcal B(X)$.
\end{definition}

\begin{theorem}[{\cite[Theorem 41]{Cortez22}; cf. \cite{Day61,Gre69,Pat88}}]
For a countable discrete group $G$, the following conditions are equivalent:
\begin{enumerate}
    \item $G$ is amenable;
    \item every continuous action of $G$ on a non-empty compact Hausdorff space admits a $G$-invariant Borel probability measure;
    \item every continuous action of $G$ on a non-empty compact metrizable space admits a $G$-invariant Borel probability measure.
\end{enumerate}
\end{theorem}

Equivalently, amenability may be formulated as a fixed-point property for continuous affine actions on non-empty compact convex subsets of locally convex topological vector spaces; see \cite{Day61,Gre69,Pat88}. This viewpoint is particularly compatible with groupoid-theoretic formulations of amenability.

\subsection{Ample groupoids and Steinberg algebras}

We now recall the basic language of ample groupoids and Steinberg algebras. A group may be regarded as a groupoid with a single unit, whereas an ample groupoid may have many units and its local symmetries are encoded by compact open bisections.

\begin{definition}[{\cite[\S 2]{Ste09}; cf. \cite[\S 2]{CFST14}}]
A \emph{groupoid} is a small category in which every arrow is invertible. We identify each object with its identity arrow and denote the set of all units by $\mathcal G^{(0)}$. Equivalently,
$$
\mathcal G^{(0)} = \{\gamma^{-1}\gamma\mid\gamma\in\mathcal G\} = \{\gamma\gamma^{-1}\mid\gamma\in\mathcal G\}.
$$
For $\gamma\in\mathcal G$, its source and range are
$$
s(\gamma)=\gamma^{-1}\gamma, \ r(\gamma)=\gamma\gamma^{-1},
$$
respectively. The set of composable pairs is
$$
\mathcal G^{(2)} = \{(\alpha,\beta)\in\mathcal G\times\mathcal G\mid s(\alpha)=r(\beta)\}.
$$
For subsets $U,V\subseteq\mathcal G$, we define
$$
UV = \{\alpha\beta\mid\alpha\in U,\ \beta\in V,\ s(\alpha)=r(\beta)\}
$$
and
$$
U^{-1}=\{\gamma^{-1}\mid\gamma\in U\}.
$$
\end{definition}

\begin{definition}[{\cite{Ren80}; \cite[\S 2]{Ste09}; cf. \cite[\S 2]{CFST14}}]
A \emph{topological groupoid} is a groupoid $\mathcal G$ endowed with a topology such that the source and range maps
$$
s,r:\mathcal G\longrightarrow\mathcal G^{(0)},
$$
the inversion map
$$
\mathcal G\longrightarrow\mathcal G, \ \gamma\longmapsto\gamma^{-1},
$$
and the multiplication map
$$
\mathcal G^{(2)}\longrightarrow\mathcal G, \ (\alpha,\beta)\longmapsto\alpha\beta,
$$
are continuous, where $\mathcal G^{(2)}$ carries the subspace topology inherited from $\mathcal G\times\mathcal G$. The unit space $\mathcal G^{(0)}$ is always equipped with the subspace topology.
\end{definition}

\begin{definition}[{\cite[\S 2]{Ste09}; cf. \cite[Remark 2.1]{CFST14}}]
A topological groupoid $\mathcal G$ is called \emph{\'{e}tale} if the source map
$$
s:\mathcal G\longrightarrow\mathcal G^{(0)}
$$
is a local homeomorphism. Equivalently, the range map
$$
r:\mathcal G\longrightarrow\mathcal G^{(0)}
$$
is a local homeomorphism.
\end{definition}

\begin{definition}[{cf. \cite[Definition 2.2]{Ste09} and \cite[\S 2]{CFST14}}]
Let $\mathcal G$ be a topological groupoid. A subset $U\subseteq\mathcal G$ is called a \emph{bisection} if the restrictions
$$
r|_U:U\longrightarrow r(U), \ s|_U:U\longrightarrow s(U),
$$
are injective. If $U$ is open, then it is called an \emph{open bisection}. When $\mathcal G$ is \'{e}tale and $U$ is an open bisection, both restrictions are homeomorphisms onto open subsets of $\mathcal G^{(0)}$.
\end{definition}

\begin{proposition}[{\cite[Propositions 2.1 and 2.3]{Ste09}; cf. \cite[Remark 2.1 and Lemma 2.2]{CFST14}}]
Let $\mathcal G$ be an \'{e}tale groupoid. Then $\mathcal G^{(0)}$ is open in $\mathcal G$, the maps $r$ and $s$ are open, and the open bisections form a basis for the topology of $\mathcal G$.
\end{proposition}

\begin{definition}[{\cite[Definition 2.4]{Ste09}}]
An \'{e}tale groupoid $\mathcal G$ is called \emph{ample} if its compact open bisections form a basis for its topology.
\end{definition}

We write
$$
\operatorname{Bis}_c(\mathcal G) = \{U\subseteq\mathcal G\mid U\text{ is a compact open bisection}\}.
$$

\begin{proposition}[{\cite[Proposition 2.3 and Definition 2.4]{Ste09}}]
Let $\mathcal G$ be a Hausdorff ample groupoid. Under setwise multiplication and inversion, $\operatorname{Bis}_c(\mathcal G)$ is an inverse semigroup. Its idempotents are precisely the compact open subsets of the unit space:
$$
E\bigl(\operatorname{Bis}_c(\mathcal G)\bigr) = \{C\subseteq\mathcal G^{(0)}\mid C\text{ is compact and open}\}.
$$
For compact open subsets $C,D\subseteq\mathcal G^{(0)}$, one has $CD=C\cap D$.
\end{proposition}

\begin{remark}[{\cite[Proposition 2.7]{Ste09}}]
If $\mathcal G$ is a Hausdorff ample groupoid, then $\mathcal G^{(0)}$ is a locally compact Hausdorff space with a basis of compact open subsets. In particular, $\mathcal G^{(0)}$ is a locally compact Boolean space. Throughout the remainder of the paper, all ample groupoids are assumed to be Hausdorff unless explicitly stated otherwise.
\end{remark}

\begin{definition}[{\cite{Ren80}; cf. \cite{ADR00}}]
Let $Y\subseteq\mathcal G^{(0)}$. The set $Y$ is called \emph{invariant} if
$$
s(\gamma)\in Y \ \Longleftrightarrow \ r(\gamma)\in Y
$$
for every $\gamma\in\mathcal G$. The \emph{reduction} of $\mathcal G$ to $Y$ is the subgroupoid
$$
\mathcal G|_Y = r^{-1}(Y)\cap s^{-1}(Y) = \{\gamma\in\mathcal G\mid r(\gamma)\in Y,\ s(\gamma)\in Y\}.
$$
\end{definition}

\begin{definition}[{\cite[\S 3.1, Example 3.2 and \S 5.2]{ClarkHazrat20}}]\label{groupoid-orbit-isotropy}
Let $u\in\mathcal G^{(0)}$. The \emph{isotropy group at $u$} is
$$
\mathcal G_u^u=\{\gamma\in\mathcal G\mid s(\gamma)=u=r(\gamma)\},
$$
and the \emph{orbit of $u$} is
$$
\mathcal O_u=r\bigl(s^{-1}(u)\bigr)=\{r(\gamma)\mid\gamma\in\mathcal G,\ s(\gamma)=u\}.
$$
The groupoid $\mathcal G$ is called \emph{transitive} if it has a single orbit, equivalently, if for every $u,v\in\mathcal G^{(0)}$ there exists $\gamma\in\mathcal G$ such that
$$
s(\gamma)=v, \ r(\gamma)=u.
$$
It is called \emph{principal} if every isotropy group is trivial, that is,
$$
\mathcal G_u^u=\{u\}
$$
for every $u\in\mathcal G^{(0)}$.
\end{definition}

\begin{definition}
A topological groupoid is called \emph{discrete} if its arrow space is endowed with the discrete topology. It is called \emph{countable} if its arrow space is countable. Thus a countable discrete groupoid is a groupoid with countably many arrows equipped with the discrete topology.
\end{definition}

Let $R$ be a non-zero commutative ring with identity $1_R$, endowed with the discrete topology.

\begin{definition}[{\cite[Definition 3.1]{CFST14}}]
Let $Z$ be a topological space. A function $f:Z\longrightarrow R$ is called \emph{locally constant} if every $z\in Z$ has an open neighbourhood on which $f$ is constant. Its support is
$$
\operatorname{supp}(f) = \overline{\{z\in Z\mid f(z)\neq 0\}}.
$$
If $f$ is locally constant, then $\{z\in Z\mid f(z)\neq 0\}$ is clopen and therefore coincides with $\operatorname{supp}(f)$.
\end{definition}

\begin{definition}[{\cite[Definition 3.1, Proposition 3.3, Definition 3.4 and Propositions 3.5--3.6]{Ste09}; cf. \cite[Definition 3.3, Lemma 3.4 and Proposition 3.7]{CFST14}}]
Let $\mathcal G$ be a Hausdorff ample groupoid. The \emph{Steinberg algebra} of $\mathcal G$ over $R$, denoted by $A_R(\mathcal G)$, is the $R$-module of all locally constant, compactly supported functions
$$
f:\mathcal G\longrightarrow R,
$$
with pointwise addition and scalar multiplication and with convolution product
$$
(f\ast g)(\gamma)=\sum_{\alpha\beta=\gamma}f(\alpha)g(\beta), \ \gamma\in\mathcal G.
$$
The sum is finite because $\mathcal G$ is \'{e}tale and both functions have compact support. Equivalently,
$$
A_R(\mathcal G)=\operatorname{span}_R\{\mathbf 1_U\mid U\in\operatorname{Bis}_c(\mathcal G)\}.
$$
Moreover, for $U,V\in\operatorname{Bis}_c(\mathcal G)$,
$$
\mathbf 1_U\ast\mathbf 1_V=\mathbf 1_{UV}.
$$
\end{definition}

\begin{proposition}[{\cite[Proposition 3.5]{Ste09}}]\label{unit-space-cut-downs}
Let $C\subseteq\mathcal G^{(0)}$ be compact and open. For every $f\in A_R(\mathcal G)$ and every $\gamma\in\mathcal G$, one has
$$
(f\ast\mathbf 1_C)(\gamma)=f(\gamma)\mathbf 1_C(s(\gamma))
$$
and
$$
(\mathbf 1_C\ast f)(\gamma)=\mathbf 1_C(r(\gamma))f(\gamma).
$$
Consequently,
$$
A_R(\mathcal G)\mathbf 1_C=\{f\in A_R(\mathcal G)\mid\operatorname{supp}(f)\subseteq s^{-1}(C)\}
$$
and
$$
\mathbf 1_CA_R(\mathcal G)=\{f\in A_R(\mathcal G)\mid\operatorname{supp}(f)\subseteq r^{-1}(C)\}.
$$
\end{proposition}

\begin{proof}
If $\beta\in\mathcal G^{(0)}$ and $\alpha\beta=\gamma$, then necessarily $\alpha=\gamma$ and $\beta=s(\gamma)$. Hence
$$
(f\ast\mathbf 1_C)(\gamma)=f(\gamma)\mathbf 1_C(s(\gamma)).
$$
The second identity is proved analogously. The two descriptions of the principal one-sided ideals follow immediately: for example, $f=f\ast\mathbf 1_C$ if and only if $f(\gamma)=0$ whenever $s(\gamma)\notin C$.
\end{proof}

A family $\mathcal E$ of idempotents in an algebra $A$ is a \emph{set of local units} if, for every finite subset $F\subseteq A$, there exists $e\in\mathcal E$ such that $ea=ae=a$ for every $a\in F$.

\begin{proposition}[{\cite[Proposition 4.11]{Ste10}; cf. \cite[\S 2]{Rigby21}}]
\label{steinberg-local-units}
Let $\mathcal G$ be a Hausdorff ample groupoid, and let $R$ be a non-zero commutative unital ring. The family
$$
\{\mathbf 1_C\mid C\subseteq\mathcal G^{(0)}\text{ is compact and open}\}
$$
is a set of local units for $A_R(\mathcal G)$. Furthermore, $A_R(\mathcal G)$ is unital if and only if $\mathcal G^{(0)}$ is compact; in that case, its identity is $\mathbf 1_{\mathcal G^{(0)}}$.
\end{proposition}

\begin{proof}
Let $f_1,\dots,f_m\in A_R(\mathcal G)$. The set
$$
K=\bigcup_{j=1}^m\bigl(s(\operatorname{supp}(f_j))\cup r(\operatorname{supp}(f_j))\bigr)
$$
is compact in the unit space. Cover $K$ by finitely many compact open subsets of $\mathcal G^{(0)}$ and let $C$ be their union. Proposition~\ref{unit-space-cut-downs} gives
$$
\mathbf 1_C\ast f_j=f_j=f_j\ast\mathbf 1_C
$$
for every $j$. Also, $\mathbf 1_C\ast\mathbf 1_C=\mathbf 1_C$, so these idempotents are local units.

If $\mathcal G^{(0)}$ is compact, the same cut-down formulas show that $\mathbf 1_{\mathcal G^{(0)}}$ is an identity. Conversely, suppose that $h\in A_R(\mathcal G)$ is an identity. For each $u\in\mathcal G^{(0)}$, choose a compact open neighborhood $C_u\subseteq\mathcal G^{(0)}$. Then
$$
h(u)=(h\ast\mathbf 1_{C_u})(u)=\mathbf 1_{C_u}(u)=1_R\neq 0.
$$
Hence $\mathcal G^{(0)}\subseteq\operatorname{supp}(h)$. The unit space is closed in the Hausdorff groupoid: it is the equalizer of the continuous maps $\gamma\mapsto\gamma$ and $\gamma\mapsto s(\gamma)$, with the source regarded as a unit arrow. It follows that $\mathcal G^{(0)}$ is a closed subset of the compact set $\operatorname{supp}(h)$ and is therefore compact.
\end{proof}

\begin{example}[{\cite[\S 3]{Ste09}}]
Let $G$ be a discrete group, regarded as a groupoid with one unit. Then every non-empty compact open bisection is a singleton $\{g\}$, with $g\in G$, and
$$
A_R(G)\cong RG.
$$
Under this identification, $\mathbf 1_{\{g\}}$ corresponds to the canonical basis element $g\in RG$. Thus Steinberg algebras genuinely extend ordinary group algebras.
\end{example}

The identity
$$
\mathbf 1_U\ast\mathbf 1_V=\mathbf 1_{UV}
$$
shows that compact open bisections play inside the Steinberg algebra the role played by group elements inside a group algebra. This observation is the guiding principle of the present work: we aim to interpret amenability-type phenomena for ample groupoids through finite-dimensional algebraic data encoded in $A_R(\mathcal G)$.

\begin{remark}
The modern theory of groupoid amenability includes topological amenability, measurewise amenability, weak amenability, exactness and weak-containment phenomena for groupoid $C^*$-algebras and representations. We focus instead on the algebraic shadow of these notions in the setting of ample groupoids and Steinberg algebras; see \cite{ADR00,AnantharamanDelaroche23,Pacheco25,SridharanKumar26}.
\end{remark}

\section{Algebraic F{\o}lner amenability}\label{algebraic-folner}

The classical F{\o}lner criterion replaces exact invariance by finite subsets whose boundary is small relative to their cardinality. Its algebraic analogue replaces finite subsets by finite-dimensional subspaces and cardinality by vector-space dimension. In the setting of ample groupoids, the compact open bisections recalled in Section~\ref{preliminaries} provide a canonical family of multipliers, because
$$
A_{\Bbbk}(\mathcal G) = \operatorname{span}_{\Bbbk} \{\mathbf 1_U\mid U\in\operatorname{Bis}_c(\mathcal G)\}
$$
and
$$
\mathbf 1_U\ast\mathbf 1_V=\mathbf 1_{UV}
$$
for all $U,V\in\operatorname{Bis}_c(\mathcal G)$. This section first recalls algebraic amenability for associative algebras and then formulates both an unanchored and a source-anchored version for Hausdorff ample groupoids.

Throughout the remainder of the paper, $\Bbbk$ denotes a field. All algebras are assumed to be associative and non-zero, but they need not be unital. All groupoids and all reductions under consideration are assumed to be non-empty. Unless otherwise specified, all vector spaces and dimensions are taken over $\Bbbk$.

\subsection{Algebraically amenable algebras}

Let $A$ be a $\Bbbk$-algebra. If $L,W\subseteq A$ are $\Bbbk$-linear subspaces, we write
$$
LW = \operatorname{span}_{\Bbbk} \{aw\mid a\in L,\ w\in W\}.
$$
For a finite-dimensional subspace $L\subseteq A$ and a non-zero finite-dimensional subspace $W\subseteq A$, the quotient
$$
\partial_L(W) = \frac{W+LW}{W}
$$
will be called the \emph{left $L$-boundary} of $W$. Thus
$$
\dim_{\Bbbk}(W+LW) = \dim_{\Bbbk}(W)+\dim_{\Bbbk}(\partial_L(W)).
$$

\begin{definition}[{\cite[Definition 3.1 and Remark 3.2]{AraLiLledoWu18}; cf. \cite[Definition 1.3 and Lemma 3.2]{Bartholdi08}}]\label{algebraically-amenable-algebra}
Let $A$ be a $\Bbbk$-algebra, let $L\subseteq A$ be finite-dimensional and let $\varepsilon>0$. A non-zero finite-dimensional subspace $W\subseteq A$ is called a \emph{left $(L,\varepsilon)$-F{\o}lner subspace} if
$$
\dim_{\Bbbk}(W+LW) \leq (1+\varepsilon)\dim_{\Bbbk}(W).
$$
Equivalently,
$$
\frac{\dim_{\Bbbk}(\partial_L(W))}{\dim_{\Bbbk}(W)} \leq \varepsilon.
$$
The algebra $A$ is called \emph{left algebraically amenable} if, for every finite-dimensional subspace $L\subseteq A$ and every $\varepsilon>0$, it admits a left $(L,\varepsilon)$-F{\o}lner subspace.
\end{definition}

The right-handed notions are obtained by replacing $LW$ with $WL$. Unless otherwise stated, all algebraic amenability conditions in this paper are left-handed.

\begin{remark}\label{proper-algebraic-amenability}
Elek's original definition for affine algebras requires an increasing exhaustive sequence of F{\o}lner subspaces; see \cite[Definition 1.1]{Elek03}. In the terminology of \cite[Definition 3.1 and Remark 3.3]{AraLiLledoWu18}, this is the stronger notion of \emph{proper algebraic amenability}. We do not impose exhaustivity in Definition~\ref{algebraically-amenable-algebra}. The distinction is essential for algebras with zero divisors: for example, the existence of a non-zero finite-dimensional left ideal already implies algebraic amenability, but need not imply proper algebraic amenability; see \cite[Example 3.7]{AraLiLledoWu18}.
\end{remark}

\begin{remark}\label{unital-plus-W}
If $A$ is unital and $1_A\in L$, then $W\subseteq LW$, and hence
$$
W+LW=LW.
$$
We retain the term $+W$ because Steinberg algebras need not be unital when the unit space is non-compact. Although they have local units, as recalled in Section~\ref{preliminaries}, a fixed finite-dimensional test space $L$ need not contain a local unit for an arbitrary F{\o}lner subspace $W$.
\end{remark}

\begin{lemma}\label{generating-set-folner}
Let $A$ be a $\Bbbk$-algebra and let $\Omega\subseteq A$ satisfy
$$
A=\operatorname{span}_{\Bbbk}(\Omega).
$$
The following conditions are equivalent:
\begin{enumerate}
    \item $A$ is left algebraically amenable;
    \item for every finite subset $\Omega_0\subseteq\Omega$ and every $\varepsilon>0$, there exists a non-zero finite-dimensional subspace $W\subseteq A$ such that
    $$
    \dim_{\Bbbk}\left(W+\sum_{\omega\in\Omega_0}\omega W\right)\leq(1+\varepsilon)\dim_{\Bbbk}(W);
    $$
    \item for every finite subset $\Omega_0\subseteq\Omega$ and every $\varepsilon>0$, there exists a non-zero finite-dimensional subspace $W\subseteq A$ such that
    $$
    \dim_{\Bbbk}(W+\omega W)\leq(1+\varepsilon)\dim_{\Bbbk}(W)
    $$
    for every $\omega\in\Omega_0$.
\end{enumerate}
\end{lemma}

\begin{proof}
Assume that $A$ is left algebraically amenable. Let $\Omega_0\subseteq\Omega$ be finite and let $\varepsilon>0$. Set
$$
L=\operatorname{span}_{\Bbbk}(\Omega_0).
$$
Choose a left $(L,\varepsilon)$-F{\o}lner subspace $W\subseteq A$. Since
$$
\sum_{\omega\in\Omega_0}\omega W\subseteq LW,
$$
we have
$$
W+\sum_{\omega\in\Omega_0}\omega W\subseteq W+LW.
$$
Therefore
$$
\dim_{\Bbbk}\left(W+\sum_{\omega\in\Omega_0}\omega W
\right)\leq(1+\varepsilon)\dim_{\Bbbk}(W).
$$
This proves $(1)\Rightarrow(2)$.

The implication $(2)\Rightarrow(3)$ is immediate, because for every $\omega\in\Omega_0$,
$$
W+\omega W\subseteq W+\sum_{\eta\in\Omega_0}\eta W.
$$

It remains to prove $(3)\Rightarrow(1)$. Let $L\subseteq A$ be finite-dimensional and let $\varepsilon>0$. Since $\Omega$ linearly spans $A$, there exists a finite subset $\Omega_0\subseteq\Omega$ such that
$$
L\subseteq\operatorname{span}_{\Bbbk}(\Omega_0).
$$
Set $q=|\Omega_0|$. If $q=0$, then $L=\{0\}$. Since $A$ is non-zero, choose $0\neq a\in A$ and set
$$
W=\operatorname{span}_{\Bbbk}\{a\}.
$$
Then $W+LW=W$, so $W$ is a left $(L,\varepsilon)$-F{\o}lner subspace.

Suppose now that $q\geq 1$. By condition $(3)$, applied with $\varepsilon/q$, there exists a non-zero finite-dimensional subspace $W\subseteq A$ such that
$$
\dim_{\Bbbk}(W+\omega W) \leq \left(1+\frac{\varepsilon}{q}\right)\dim_{\Bbbk}(W)
$$
for every $\omega\in\Omega_0$. Since
$$
W+LW \subseteq W+\sum_{\omega\in\Omega_0}\omega W,
$$
it is enough to estimate the dimension of the latter space. Passing to the quotient by $W$ and using subadditivity of dimension gives
$$
\begin{aligned}
&\dim_{\Bbbk}\left(W+\sum_{\omega\in\Omega_0}\omega W\right)-\dim_{\Bbbk}(W) \\
&\qquad= \dim_{\Bbbk}\left(\frac{W+\sum_{\omega\in\Omega_0}\omega W}{W}\right)\\
&\qquad\leq\sum_{\omega\in\Omega_0}\dim_{\Bbbk}\left(\frac{W+\omega W}{W}\right)
\\
&\qquad=\sum_{\omega\in\Omega_0}\left(\dim_{\Bbbk}(W+\omega W)-\dim_{\Bbbk}(W)\right)
\\
&\qquad\leq q \frac{\varepsilon}{q}\dim_{\Bbbk}(W)=\varepsilon\dim_{\Bbbk}(W).
\end{aligned}
$$
Consequently,
$$
\dim_{\Bbbk}(W+LW)\leq(1+\varepsilon)\dim_{\Bbbk}(W).
$$
Thus $A$ is left algebraically amenable.
\end{proof}

We record explicitly the group-algebra criterion that will be used in the one-unit calibration below.

\begin{theorem}[Bartholdi's group-algebra criterion, {\cite[Corollary 1.5]{Bartholdi08}}]\label{bartholdi-group-algebra}
Let $G$ be a discrete group and let $\Bbbk$ be a field. The following conditions are equivalent:
\begin{enumerate}
    \item $G$ is amenable;
    \item the group algebra $\Bbbk G$ is left algebraically amenable;
    \item the group algebra $\Bbbk G$ is right algebraically amenable.
\end{enumerate}
\end{theorem}

\begin{proof}
Bartholdi proves the equivalence between amenability of $G$ and right algebraic amenability of $\Bbbk G$; see \cite[Definition 1.3 and Corollary 1.5]{Bartholdi08}. To pass between the right and left formulations, define
$$
\iota:\Bbbk G\longrightarrow\Bbbk G, \ \iota\left(\sum_{g\in G}\lambda_g g\right) = \sum_{g\in G}\lambda_g g^{-1}.
$$
The map $\iota$ is a $\Bbbk$-linear anti-automorphism, so
$$
\iota(ab)=\iota(b)\iota(a)
$$
for all $a,b\in\Bbbk G$.

Suppose that $\Bbbk G$ is right algebraically amenable. Let $L\subseteq\Bbbk G$ be finite-dimensional and let $\varepsilon>0$. Choose a non-zero finite-dimensional subspace $W\subseteq\Bbbk G$ such that
$$
\dim_{\Bbbk}\bigl(W+W\iota(L)\bigr) \leq (1+\varepsilon)\dim_{\Bbbk}(W).
$$
Since $\iota$ is a linear isomorphism and reverses products,
$$
\iota\bigl(W+W\iota(L)\bigr) = \iota(W)+L\iota(W).
$$
Therefore
$$
\dim_{\Bbbk}\bigl(\iota(W)+L\iota(W)\bigr) \leq (1+\varepsilon)\dim_{\Bbbk}\bigl(\iota(W)\bigr),
$$
which proves left algebraic amenability. The reverse implication follows from the same argument because $\iota^2$ is the identity. Thus the left and right conditions are equivalent, and Bartholdi's result completes the proof.
\end{proof}

\subsection{Weak bisectional F{\o}lner amenability}

Let $\mathcal G$ be a Hausdorff ample groupoid and set
$$
A=A_{\Bbbk}(\mathcal G).
$$
By the results recalled in Section~\ref{preliminaries}, the family
$$
\{\mathbf 1_U\mid U\in\operatorname{Bis}_c(\mathcal G)\}
$$
linearly spans $A$. This leads to the following intrinsic test on the canonical bisection generators.

\begin{definition}\label{weak-bisectional-folner}
Let $\mathcal E\subseteq\operatorname{Bis}_c(\mathcal G)$ be finite and let $\varepsilon>0$. A non-zero finite-dimensional subspace $W\subseteq A_{\Bbbk}(\mathcal G)$ is called a \emph{weak $(\mathcal E,\varepsilon)$-bisectional F{\o}lner subspace} if
$$
\dim_{\Bbbk}(W+\mathbf 1_U W)\leq(1+\varepsilon)\dim_{\Bbbk}(W)
$$
for every $U\in\mathcal E$.
\end{definition}

\begin{definition}\label{weak-bisectional-amenability}
A Hausdorff ample groupoid $\mathcal G$ is called \emph{weakly left bisectionally algebraically amenable over $\Bbbk$} if, for every finite subset $\mathcal E\subseteq\operatorname{Bis}_c(\mathcal G)$ and every $\varepsilon>0$, there exists a weak $(\mathcal E,\varepsilon)$-bisectional F{\o}lner subspace of $A_{\Bbbk}(\mathcal G)$.
\end{definition}

\begin{proposition}\label{weak-bisectional-equivalence}
Let $\mathcal G$ be a Hausdorff ample groupoid. Then $\mathcal G$ is weakly left bisectionally algebraically amenable over $\Bbbk$ if and only if its Steinberg algebra $A_{\Bbbk}(\mathcal G)$ is left algebraically amenable.
\end{proposition}

\begin{proof}
Assume first that $A_{\Bbbk}(\mathcal G)$ is left algebraically amenable. Let $\mathcal E\subseteq\operatorname{Bis}_c(\mathcal G)$ be finite and let $\varepsilon>0$. Set
$$
L=\operatorname{span}_{\Bbbk}\{\mathbf 1_U\mid U\in\mathcal E\}.
$$
Choose a left $(L,\varepsilon)$-F{\o}lner subspace $W\subseteq A_{\Bbbk}(\mathcal G)$. For every $U\in\mathcal E$,
$$
W+\mathbf 1_U W\subseteq W+LW.
$$
Hence
$$
\dim_{\Bbbk}(W+\mathbf 1_U W)\leq\dim_{\Bbbk}(W+LW)\leq(1+\varepsilon)\dim_{\Bbbk}(W).
$$
Thus $W$ is a weak $(\mathcal E,\varepsilon)$-bisectional F{\o}lner subspace.

Conversely, suppose that $\mathcal G$ is weakly left bisectionally algebraically amenable. Apply Lemma~\ref{generating-set-folner} to
$$
\Omega=\{\mathbf 1_U\mid U\in\operatorname{Bis}_c(\mathcal G)\}.
$$
The family $\Omega$ spans $A_{\Bbbk}(\mathcal G)$, and Definition~\ref{weak-bisectional-amenability} is precisely condition $(3)$ of that lemma. Therefore $A_{\Bbbk}(\mathcal G)$ is left algebraically amenable.
\end{proof}

\begin{remark}\label{weak-condition-only-sees-algebra}
Proposition~\ref{weak-bisectional-equivalence} shows that weak bisectional amenability is exactly algebraic amenability of the Steinberg algebra expressed in its canonical bisection generators. It does not force the F{\o}lner subspace to interact with a prescribed region of the unit space. For groupoids with several components, this permits a F{\o}lner subspace to concentrate in only one component. The source-anchored condition introduced next prevents such avoidance when the prescribed anchor is contained in the unit space of that component. It does not, in general, prevent escape in the range direction.
\end{remark}

\subsection{Source-anchored algebraic amenability}

Let $C\subseteq\mathcal G^{(0)}$ be non-empty, compact and open. By Proposition~\ref{unit-space-cut-downs},
$$
A_{\Bbbk}(\mathcal G)\mathbf 1_C=\{f\in A_{\Bbbk}(\mathcal G)\mid\operatorname{supp}(f)\subseteq s^{-1}(C)\}.
$$
Moreover, $A_{\Bbbk}(\mathcal G)\mathbf 1_C$ is a left ideal of $A_{\Bbbk}(\mathcal G)$, because
$$
a(b\mathbf 1_C)=(ab)\mathbf 1_C
$$
for all $a,b\in A_{\Bbbk}(\mathcal G)$. Thus it is stable under all left multipliers appearing in the F{\o}lner condition.

\begin{definition}\label{anchored-bisectional-folner}
Let $C\subseteq\mathcal G^{(0)}$ be non-empty, compact and open, let $\mathcal E\subseteq\operatorname{Bis}_c(\mathcal G)$ be finite and let $\varepsilon>0$. A non-zero finite-dimensional subspace
$$
W\subseteq A_{\Bbbk}(\mathcal G)\mathbf 1_C
$$
is called a \emph{source-anchored $(C,\mathcal E,\varepsilon)$-bisectional F{\o}lner subspace} if
$$
\dim_{\Bbbk}(W+\mathbf 1_U W)\leq(1+\varepsilon)\dim_{\Bbbk}(W)
$$
for every $U\in\mathcal E$.
\end{definition}

\begin{definition}\label{anchored-algebraic-amenability}
A Hausdorff ample groupoid $\mathcal G$ is called \emph{source-anchored left algebraically amenable over $\Bbbk$} if, for every non-empty compact open subset $C\subseteq\mathcal G^{(0)}$, every finite subset $\mathcal E\subseteq\operatorname{Bis}_c(\mathcal G)$ and every $\varepsilon>0$, there exists a source-anchored $(C,\mathcal E,\varepsilon)$-bisectional F{\o}lner subspace.
\end{definition}

Unless explicitly qualified by the word ``weakly'', the expression \emph{algebraically amenable groupoid} will mean source-anchored left algebraically amenable in the remainder of the paper.

The next result gives the finite-dimensional multiplier formulation of our definition.

\begin{proposition}\label{anchored-finite-dimensional-test}
Let $\mathcal G$ be a Hausdorff ample groupoid. The following conditions are equivalent:
\begin{enumerate}
    \item $\mathcal G$ is source-anchored left algebraically amenable over $\Bbbk$;
    \item for every non-empty compact open subset $C\subseteq\mathcal G^{(0)}$, every finite-dimensional subspace $L\subseteq A_{\Bbbk}(\mathcal G)$ and every $\varepsilon>0$, there exists a non-zero finite-dimensional subspace
    $$
    W\subseteq A_{\Bbbk}(\mathcal G)\mathbf 1_C
    $$
    such that
    $$
    \dim_{\Bbbk}(W+LW)\leq(1+\varepsilon)\dim_{\Bbbk}(W).
    $$
\end{enumerate}
\end{proposition}

\begin{proof}
Assume first that condition $(2)$ holds. Let $C\subseteq\mathcal G^{(0)}$ be non-empty, compact and open, let $\mathcal E\subseteq\operatorname{Bis}_c(\mathcal G)$ be finite and let $\varepsilon>0$. Set
$$
L=\operatorname{span}_{\Bbbk}\{\mathbf 1_U\mid U\in\mathcal E\}.
$$
By condition $(2)$, there exists a non-zero finite-dimensional subspace
$$
W\subseteq A_{\Bbbk}(\mathcal G)\mathbf 1_C
$$
such that
$$
\dim_{\Bbbk}(W+LW)\leq(1+\varepsilon)\dim_{\Bbbk}(W).
$$
For every $U\in\mathcal E$,
$$
W+\mathbf 1_U W\subseteq W+LW.
$$
Therefore $W$ is a source-anchored $(C,\mathcal E,\varepsilon)$-bisectional F{\o}lner subspace. Hence condition $(1)$ holds.

Conversely, assume condition $(1)$. Let $C\subseteq\mathcal G^{(0)}$ be non-empty, compact and open, let $L\subseteq A_{\Bbbk}(\mathcal G)$ be finite-dimensional and let $\varepsilon>0$. Since the characteristic functions of compact open bisections linearly span $A_{\Bbbk}(\mathcal G)$, there exists a finite subset $\mathcal E\subseteq\operatorname{Bis}_c(\mathcal G)$ such that
$$
L\subseteq\operatorname{span}_{\Bbbk}\{\mathbf 1_U\mid U\in\mathcal E\}.
$$
Set $q=|\mathcal E|$. If $q=0$, then $L=\{0\}$. Taking
$$
W=\operatorname{span}_{\Bbbk}\{\mathbf 1_C\}\subseteq A_{\Bbbk}(\mathcal G)\mathbf 1_C
$$
gives $W+LW=W$.

Suppose that $q\geq 1$. By condition $(1)$, applied with $\varepsilon/q$, there exists a non-zero finite-dimensional subspace
$$
W\subseteq A_{\Bbbk}(\mathcal G)\mathbf 1_C
$$
such that
$$
\dim_{\Bbbk}(W+\mathbf 1_U W)\leq\left(1+\frac{\varepsilon}{q}\right)\dim_{\Bbbk}(W)
$$
for every $U\in\mathcal E$. Since $A_{\Bbbk}(\mathcal G)\mathbf 1_C$ is a left ideal, all spaces in the following computation remain inside that ideal. As in the proof of Lemma~\ref{generating-set-folner},
$$
W+LW\subseteq W+\sum_{U\in\mathcal E}\mathbf 1_U W
$$
and
$$
\begin{aligned}
&\dim_{\Bbbk} \left(W+\sum_{U\in\mathcal E}\mathbf 1_U W\right)-\dim_{\Bbbk}(W) \\
&\qquad\leq\sum_{U\in\mathcal E}\left(\dim_{\Bbbk}(W+\mathbf 1_U W)-\dim_{\Bbbk}(W)\right) \\
&\qquad\leq q\frac{\varepsilon}{q}\dim_{\Bbbk}(W)=\varepsilon\dim_{\Bbbk}(W).
\end{aligned}
$$
Consequently,
$$
\dim_{\Bbbk}(W+LW)\leq(1+\varepsilon)\dim_{\Bbbk}(W).
$$
Thus condition $(2)$ holds.
\end{proof}

\begin{remark}\label{anchored-module-interpretation}
The source-anchored condition can be expressed in the language of algebraically amenable modules. For a non-zero left $A$-module $M$, use the left-handed analogue of \cite[Definition 1.3]{Bartholdi08}: for every finite-dimensional $L\subseteq A$ and every $\varepsilon>0$, there must exist a non-zero finite-dimensional $W\subseteq M$ with
$$
\dim_{\Bbbk}(W+LW)\leq(1+\varepsilon)\dim_{\Bbbk}(W).
$$
By Proposition~\ref{anchored-finite-dimensional-test}, a Hausdorff ample groupoid $\mathcal G$ is source-anchored algebraically amenable exactly when every left module $A_{\Bbbk}(\mathcal G)\mathbf 1_C$, with $C$ non-empty compact open in $\mathcal G^{(0)}$, has this property. Thus the anchoring selects a canonical family of modules determined by the unit space; the underlying module F{\o}lner condition is not a new notion of amenability for arbitrary algebras.
\end{remark}

\begin{proposition}\label{anchored-implies-weak}
Every source-anchored left algebraically amenable Hausdorff ample groupoid is weakly left bisectionally algebraically amenable.
\end{proposition}

\begin{proof}
Let $\mathcal G$ be source-anchored left algebraically amenable. Since $\mathcal G$ is non-empty and ample, its unit space contains a non-empty compact open subset $C$. Given a finite subset $\mathcal E\subseteq\operatorname{Bis}_c(\mathcal G)$ and $\varepsilon>0$, any source-anchored $(C,\mathcal E,\varepsilon)$-bisectional F{\o}lner subspace is, after forgetting the containment in $A_{\Bbbk}(\mathcal G)\mathbf 1_C$, a weak $(\mathcal E,\varepsilon)$-bisectional F{\o}lner subspace.
\end{proof}

\begin{remark}\label{weak-not-anchored}
The converse to Proposition~\ref{anchored-implies-weak} fails in general. As will be shown later, weak bisectional amenability may be witnessed inside a single amenable component even when another component is non-amenable. Source anchoring prevents this phenomenon by requiring a F{\o}lner subspace over every non-empty compact open region of the unit space.
\end{remark}

\subsection{The one-unit case}

Let $G$ be a discrete group regarded as a groupoid with one unit, denoted by $e_G$. Then
$$
G^{(0)}=\{e_G\},
$$
and every non-empty compact open bisection is a singleton $\{g\}$ with $g\in G$. Moreover,
$$
A_{\Bbbk}(G)\cong\Bbbk G,\ \mathbf 1_{\{g\}}\longleftrightarrow\delta_g.
$$
Since $\mathbf 1_{\{e_G\}}$ is the identity of $\Bbbk G$,
$$
A_{\Bbbk}(G)\mathbf 1_{\{e_G\}}=A_{\Bbbk}(G).
$$
Thus the weak and source-anchored conditions coincide in the one-unit case.

\begin{theorem}\label{group-case}
Let $G$ be a countable discrete group and let $\Bbbk$ be a field. The following conditions are equivalent:
\begin{enumerate}
    \item $G$ is amenable;
    \item the group algebra $\Bbbk G$ is left algebraically amenable;
    \item the group algebra $\Bbbk G$ is right algebraically amenable;
    \item the one-unit groupoid associated to $G$ is weakly left bisectionally algebraically amenable over $\Bbbk$;
    \item the one-unit groupoid associated to $G$ is source-anchored left algebraically amenable over $\Bbbk$.
\end{enumerate}
\end{theorem}

\begin{proof}
The equivalence of conditions $(1)$--$(3)$ is Theorem~\ref{bartholdi-group-algebra}. The equivalence $(2)\Longleftrightarrow(4)$ follows from Proposition~\ref{weak-bisectional-equivalence}. Since the unit space consists of the single compact open set $\{e_G\}$ and
$$
A_{\Bbbk}(G)\mathbf 1_{\{e_G\}}=A_{\Bbbk}(G),
$$
conditions $(4)$ and $(5)$ are equivalent.

For completeness, we give the direct construction of the F{\o}lner subspaces in the implication $(1)\Rightarrow(5)$. Let $\mathcal E\subseteq\operatorname{Bis}_c(G)$ be finite and let $\varepsilon>0$. Discarding the empty bisection, there is a finite subset $K\subseteq G$ such that
$$
\mathcal E\setminus\{\varnothing\}=\{\{g\}\mid g\in K\}.
$$
Apply the left F{\o}lner condition to $K^{-1}$ and $\varepsilon$. There exists a non-empty finite subset $F\subseteq G$ such that
$$
\frac{|F\setminus g^{-1}F|}{|F|}<\varepsilon
$$
for every $g\in K$. Since left multiplication by $g$ is a bijection,
$$
|gF\setminus F|=|F\setminus g^{-1}F|.
$$
Therefore
$$
\frac{|F\cup gF|}{|F|}=1+\frac{|gF\setminus F|}{|F|}<1+\varepsilon
$$
for every $g\in K$.

Define
$$
W_F=\operatorname{span}_{\Bbbk}\{\delta_t\mid t\in F\}\subseteq\Bbbk G.
$$
The canonical group elements $\{\delta_t\mid t\in G\}$ form a $\Bbbk$-basis of $\Bbbk G$, and
$$
\delta_g W_F=\operatorname{span}_{\Bbbk}\{\delta_{gt}\mid t\in F\}.
$$
Consequently,
$$
W_F+\delta_g W_F=\operatorname{span}_{\Bbbk}\{\delta_t\mid t\in F\cup gF\},
$$
and hence
$$
\begin{aligned}
\dim_{\Bbbk}(W_F+\delta_g W_F)&=|F\cup gF|\\
&<(1+\varepsilon)|F|\\
&=(1+\varepsilon)\dim_{\Bbbk}(W_F).
\end{aligned}
$$
For the empty bisection, multiplication by $\mathbf 1_{\varnothing}=0$ leaves $W_F$ unchanged. Finally,
$$
W_F\subseteq A_{\Bbbk}(G)\mathbf 1_{\{e_G\}}=A_{\Bbbk}(G),
$$
so $W_F$ is a source-anchored $(\{e_G\},\mathcal E,\varepsilon)$-bisectional F{\o}lner subspace.
\end{proof}

\begin{corollary}\label{field-independence-group-case}
Let $G$ be a countable discrete group. If the one-unit groupoid associated to $G$ is source-anchored left algebraically amenable over one field, then it is source-anchored left algebraically amenable over every field.
\end{corollary}

\begin{proof}
By Theorem~\ref{group-case}, source-anchored algebraic amenability over one field is equivalent to amenability of $G$, and amenability of $G$ is independent of the coefficient field. Applying the same theorem to any other field gives the conclusion.
\end{proof}

\begin{remark}
Theorem~\ref{group-case} is the basic consistency test for the groupoid notion developed here. It shows that source anchoring introduces no additional restriction in the one-unit setting, while for groupoids with many units it prevents F{\o}lner subspaces from avoiding prescribed compact open regions of the unit space.
\end{remark}

\section{Permanence properties}\label{permanence}

Section~\ref{algebraic-folner} introduced two algebraic F{\o}lner properties for Hausdorff ample groupoids. The weak property is exactly algebraic amenability of the Steinberg algebra, whereas the source-anchored property requires F{\o}lner subspaces over every non-empty compact open region of the unit space. Throughout this section, the expression \emph{algebraically amenable groupoid} always means source-anchored left algebraically amenable over the fixed field $\Bbbk$, in the sense of Definition~\ref{anchored-algebraic-amenability}.

The classical permanence properties recalled in Section~\ref{preliminaries} suggest asking whether the groupoid property is preserved by subgroupoids, disjoint unions, products and directed unions. The answer for arbitrary subgroupoids is negative. We first isolate the natural positive class of subgroupoids and then give two counterexamples that explain the role, and the limitations, of source anchoring.

\subsection{Isomorphism invariance}

\begin{proposition}\label{isomorphism-invariance}
Let $\Phi:\mathcal G\longrightarrow\mathcal H$ be an isomorphism of topological groupoids between Hausdorff ample groupoids. Then $\mathcal G$ is algebraically amenable over $\Bbbk$ if and only if $\mathcal H$ is algebraically amenable over $\Bbbk$.
\end{proposition}

\begin{proof}
The groupoid isomorphism $\Phi$ induces a $\Bbbk$-algebra isomorphism
$$
\Phi_{\#}:A_{\Bbbk}(\mathcal G) \longrightarrow A_{\Bbbk}(\mathcal H), \ \Phi_{\#}(f)=f\circ\Phi^{-1}.
$$
For every compact open bisection $U\subseteq\mathcal G$,
$$
\Phi_{\#}(\mathbf 1_U)=\mathbf 1_{\Phi(U)}.
$$
In particular, if $C\subseteq\mathcal G^{(0)}$ is compact and open, then
$$
\Phi_{\#}\left(A_{\Bbbk}(\mathcal G)\mathbf 1_C\right)= A_{\Bbbk}(\mathcal H)\mathbf 1_{\Phi(C)}.
$$
Moreover, $\Phi_{\#}$ preserves products, finite-dimensionality and vector-space dimension. Therefore, the image under $\Phi_{\#}$ of a source-anchored F{\o}lner subspace for $(C,L,\varepsilon)$ is a source-anchored F{\o}lner subspace for
$$
\left(\Phi(C),\Phi_{\#}(L),\varepsilon\right).
$$
Applying the same argument to $\Phi^{-1}$ proves the converse.
\end{proof}

\subsection{Open invariant reductions and subgroupoids}

For the purposes of this paper, we use the following terminology.

\begin{definition}\label{source-saturated-subgroupoid}
Let $\mathcal H\subseteq\mathcal G$ be a subgroupoid and put $Y=\mathcal H^{(0)}$. We say that $\mathcal H$ is \emph{source-saturated in $\mathcal G$} if
$$
s_{\mathcal G}^{-1}(Y)\subseteq\mathcal H.
$$
\end{definition}

Since every arrow of $\mathcal H$ has source and range in $Y$, one always has
$$
\mathcal H\subseteq s_{\mathcal G}^{-1}(Y)\cap r_{\mathcal G}^{-1}(Y).
$$
The source-saturated condition is therefore very restrictive, and it has a simple intrinsic description.

\begin{lemma}\label{source-saturated-reduction}
Let $\mathcal H\subseteq\mathcal G$ be a subgroupoid and put $Y=\mathcal H^{(0)}$. The following conditions are equivalent:
\begin{enumerate}
    \item $\mathcal H$ is source-saturated in $\mathcal G$;
    \item $Y$ is invariant in $\mathcal G$ and
    $$
    \mathcal H=\mathcal G|_Y.
    $$
\end{enumerate}
\end{lemma}

\begin{proof}
Assume first that $\mathcal H$ is source-saturated. Let $\gamma\in\mathcal G$ satisfy $s(\gamma)\in Y$. Then $\gamma\in\mathcal H$, and hence $r(\gamma)\in\mathcal H^{(0)}=Y$. Conversely, if $r(\gamma)\in Y$, then
$$
s(\gamma^{-1})=r(\gamma)\in Y.
$$
Source saturation gives $\gamma^{-1}\in\mathcal H$, and therefore $\gamma\in\mathcal H$ and $s(\gamma)\in Y$. Thus
$$
s(\gamma)\in Y\ \Longleftrightarrow\ r(\gamma)\in Y,
$$
so $Y$ is invariant.

Every arrow of $\mathcal H$ has source and range in $Y$, and hence
$$
\mathcal H\subseteq\mathcal G|_Y.
$$
Conversely, if $\gamma\in\mathcal G|_Y$, then $s(\gamma)\in Y$, so source saturation implies $\gamma\in\mathcal H$. Therefore
$$
\mathcal H=\mathcal G|_Y.
$$

The converse is immediate. If $Y$ is invariant and $\mathcal H=\mathcal G|_Y$, then every $\gamma\in\mathcal G$ with $s(\gamma)\in Y$ also satisfies $r(\gamma)\in Y$, and hence $\gamma\in\mathcal G|_Y=\mathcal H$.
\end{proof}

\begin{lemma}\label{open-subgroupoid-extension-by-zero}
Let $\mathcal H$ be an open subgroupoid of a Hausdorff ample groupoid $\mathcal G$, equipped with the subspace topology. Extension by zero defines an injective algebra homomorphism
$$
j:A_{\Bbbk}(\mathcal H)\longrightarrow A_{\Bbbk}(\mathcal G).
$$
If $Y=\mathcal H^{(0)}$ is invariant in $\mathcal G$ and $\mathcal H=\mathcal G|_Y$, then, for every compact open $C\subseteq Y$,
$$
j\bigl(A_{\Bbbk}(\mathcal H)\mathbf 1_C\bigr)=A_{\Bbbk}(\mathcal G)\mathbf 1_C.
$$
\end{lemma}

\begin{proof}
For $f\in A_{\Bbbk}(\mathcal H)$, define $j(f)$ to agree with $f$ on $\mathcal H$ and to vanish elsewhere. Its support is compact in $\mathcal G$, since the inclusion of $\mathcal H$ is continuous. It is also closed in $\mathcal G$, since $\mathcal G$ is Hausdorff. Local constancy holds on the open set $\mathcal H$, and every point outside $\mathcal H$ has an open neighborhood disjoint from this support. Thus $j(f)$ is locally constant and compactly supported.

A non-zero summand in $(j(f)\ast j(g))(\gamma)$ has both factors in $\mathcal H$. Their product belongs to $\mathcal H$, since $\mathcal H$ is a subgroupoid. For $\gamma\in\mathcal H$, the remaining summands are exactly those defining $(f\ast g)(\gamma)$; for $\gamma\notin\mathcal H$, there are none. Hence $j(f)\ast j(g)=j(f\ast g)$. Linearity and injectivity are immediate.

Under the additional hypotheses, every arrow with source in $C$ has range in $Y$ and therefore belongs to $\mathcal H$. Proposition~\ref{unit-space-cut-downs} gives the asserted equality of left ideals. For the reverse inclusion in that equality, restriction of a function supported in $s^{-1}(C)$ to $\mathcal H$ remains locally constant and compactly supported.
\end{proof}

The next proposition is the basic subgroupoid permanence result for our definition. Notice that openness, rather than clopenness, is sufficient.

\begin{proposition}\label{open-invariant-reduction}
Let $\mathcal G$ be an algebraically amenable Hausdorff ample groupoid, and let $Y\subseteq\mathcal G^{(0)}$ be non-empty, open and invariant. Then the reduction $\mathcal G|_Y$ is algebraically amenable over $\Bbbk$.
\end{proposition}

\begin{proof}
Set
$$
\mathcal H=\mathcal G|_Y.
$$
Since $Y$ is open and invariant, $\mathcal H$ is an open Hausdorff ample subgroupoid of $\mathcal G$. Extension by zero identifies $A_{\Bbbk}(\mathcal H)$ with a subalgebra of $A_{\Bbbk}(\mathcal G)$.

Let $C\subseteq\mathcal H^{(0)}=Y$ be non-empty, compact and open in $Y$. Because $Y$ is open in $\mathcal G^{(0)}$, the set $C$ is also open in $\mathcal G^{(0)}$. It is compact in $\mathcal G^{(0)}$ because the inclusion $Y\hookrightarrow\mathcal G^{(0)}$ is continuous. Hence $C$ is a non-empty compact open subset of $\mathcal G^{(0)}$.

Let $L\subseteq A_{\Bbbk}(\mathcal H)$ be finite-dimensional and let $\varepsilon>0$. We regard $L$ as a subspace of $A_{\Bbbk}(\mathcal G)$ by extension by zero. Since $Y$ is invariant and $C\subseteq Y$, every arrow of $\mathcal G$ with source in $C$ belongs to $\mathcal H$. Consequently,
$$
s_{\mathcal G}^{-1}(C)=s_{\mathcal H}^{-1}(C).
$$
Using Proposition~\ref{unit-space-cut-downs}, we obtain
$$
A_{\Bbbk}(\mathcal G)\mathbf 1_C=A_{\Bbbk}(\mathcal H)\mathbf 1_C.
$$

Since $\mathcal G$ is algebraically amenable, Proposition~\ref{anchored-finite-dimensional-test} gives a non-zero finite-dimensional subspace
$$
W\subseteq A_{\Bbbk}(\mathcal G)\mathbf 1_C
$$
such that
$$
\dim_{\Bbbk}(W+LW)\leq(1+\varepsilon)\dim_{\Bbbk}(W).
$$
The equality of the anchored ideals shows that
$$
W\subseteq A_{\Bbbk}(\mathcal H)\mathbf 1_C.
$$
Thus the same subspace $W$ verifies the source-anchored F{\o}lner condition for $\mathcal H$. Therefore $\mathcal H$ is algebraically amenable.
\end{proof}

\begin{corollary}\label{source-saturated-subgroupoid-col}
Let $\mathcal H$ be a non-empty open source-saturated subgroupoid of an algebraically amenable Hausdorff ample groupoid $\mathcal G$. Then $\mathcal H$ is algebraically amenable over $\Bbbk$.
\end{corollary}

\begin{proof}
By Lemma~\ref{source-saturated-reduction}, there is an open invariant subset $Y\subseteq\mathcal G^{(0)}$ such that
$$
\mathcal H=\mathcal G|_Y.
$$
The conclusion follows from Proposition~\ref{open-invariant-reduction}.
\end{proof}

The first counterexample separates the weak and source-anchored notions introduced in Section~\ref{algebraic-folner}.

\begin{proposition}\label{weak-not-anchored-counterexample}
There exists a Hausdorff ample groupoid that is weakly left bisectionally algebraically amenable over $\Bbbk$ but is not source-anchored left algebraically amenable over $\Bbbk$.
\end{proposition}

\begin{proof}
Let $F_2$ be the free group on two generators, regarded as a one-unit discrete groupoid, and let $\mathcal T$ be the trivial groupoid with one unit $u$. Consider the topological disjoint union
$$
\mathcal G=F_2\sqcup\mathcal T.
$$
This is a Hausdorff ample groupoid.

Set
$$
W=\operatorname{span}_{\Bbbk}\{\mathbf 1_{\{u\}}\}.
$$
Let $U\in\operatorname{Bis}_c(\mathcal G)$. The intersection $U\cap\mathcal T$ is either empty or equal to $\{u\}$. Hence
$$
\mathbf 1_UW=\begin{cases}\{0\}, & U\cap\mathcal T=\varnothing,\\ W, & U\cap\mathcal T=\{u\}.\end{cases}
$$
Therefore
$$
W+\mathbf 1_UW=W
$$
for every $U\in\operatorname{Bis}_c(\mathcal G)$. The same one-dimensional subspace works for every finite family of compact open bisections and every $\varepsilon>0$. Thus $\mathcal G$ is weakly left bisectionally algebraically amenable.

Suppose, toward a contradiction, that $\mathcal G$ were source-anchored left algebraically amenable. The unit space of the $F_2$-component is an open invariant subset of $\mathcal G^{(0)}$. Proposition~\ref{open-invariant-reduction} would imply that the reduction to this component, namely the one-unit groupoid $F_2$, is source-anchored left algebraically amenable. By Theorem~\ref{group-case}, this would imply that $F_2$ is an amenable group, contradicting the classical non-amenability of $F_2$. Hence $\mathcal G$ is not source-anchored left algebraically amenable.
\end{proof}

Source anchoring prevents a F{\o}lner subspace from moving to a different component. It does not, however, force the subspace to remain inside an arbitrary subgroupoid with the same anchored source region. The following example is the fundamental obstruction.

\begin{proposition}\label{arbitrary-subgroupoids-fail}
Source-anchored left algebraic amenability is not inherited by arbitrary Hausdorff ample subgroupoids. In fact, there exists a countable discrete algebraically amenable groupoid containing an open subgroupoid isomorphic to the non-amenable group $F_2$.
\end{proposition}

\begin{proof}
Let $I=\mathbb N$ and let $\Gamma=F_2$. Consider the discrete groupoid
$$
\mathcal G=I\times\Gamma\times I.
$$
Its unit space is identified with $I$ through
$$
i\longleftrightarrow(i,e_{\Gamma},i).
$$
The range and source maps are
$$
r(i,g,j)=i, \ s(i,g,j)=j,
$$
and the multiplication and inversion are given by
$$
(i,g,j)(j,h,k)=(i,gh,k) \text{ and }(i,g,j)^{-1}=(j,g^{-1},i).
$$
Since $\mathcal G$ is discrete, it is Hausdorff and ample, and every compact subset of $\mathcal G$ is finite.

We first prove that $\mathcal G$ is source-anchored left algebraically amenable. Let $C\subseteq\mathcal G^{(0)}=I$ be non-empty, compact and open. Thus $C$ is a non-empty finite subset of $I$. Let $L\subseteq A_{\Bbbk}(\mathcal G)$ be finite-dimensional and let $\varepsilon>0$.

Choose a basis $f_1,\dots,f_n$ of $L$, where the list is allowed to be empty when $L=\{0\}$. Since every $f_j$ has finite support, the set
$$
S=\bigcup_{j=1}^n\operatorname{supp}(f_j)
$$
is finite. Hence
$$
s(S)=\{s(\alpha)\mid\alpha\in S\}
$$
is a finite subset of $I$. Choose
$$
i_0\in I\setminus s(S)
$$
and choose $c_0\in C$. Set
$$
\beta=(i_0,e_{\Gamma},c_0)
$$
and
$$
W=\operatorname{span}_{\Bbbk}\{\mathbf 1_{\{\beta\}}\}.
$$
Because $s(\beta)=c_0\in C$, Proposition~\ref{unit-space-cut-downs} gives
$$
W\subseteq A_{\Bbbk}(\mathcal G)\mathbf 1_C.
$$

We claim that $LW=\{0\}$. Indeed, if $\alpha\in S$, then the product $\alpha\beta$ can be defined only if
$$
s(\alpha)=r(\beta)=i_0.
$$
This is impossible because $i_0\notin s(S)$. Therefore
$$
f_j\ast\mathbf 1_{\{\beta\}}=0
$$
for every $j$, and hence $LW=\{0\}$. It follows that
$$
W+LW=W
$$
and consequently
$$
\dim_{\Bbbk}(W+LW)=\dim_{\Bbbk}(W)\leq(1+\varepsilon)\dim_{\Bbbk}(W).
$$
Thus $\mathcal G$ is source-anchored left algebraically amenable.

Fix $u\in I$ and consider the isotropy subgroupoid
$$
\mathcal H=\{u\}\times\Gamma\times\{u\}.
$$
Since $\mathcal G$ is discrete, $\mathcal H$ is an open Hausdorff ample subgroupoid. It is a one-unit groupoid naturally isomorphic to $\Gamma=F_2$. By Theorem~\ref{group-case}, $\mathcal H$ is source-anchored left algebraically amenable if and only if $F_2$ is amenable. Since $F_2$ is not amenable, $\mathcal H$ is not source-anchored left algebraically amenable.
\end{proof}

\begin{remark}\label{range-escape}
Proposition~\ref{arbitrary-subgroupoids-fail} shows precisely what source anchoring controls. The condition
$$
W\subseteq A_{\Bbbk}(\mathcal G)\mathbf 1_C
$$
forces the source of the support of $W$ to lie in $C$, but it does not constrain its range. In the preceding example, the vector $\mathbf 1_{\{\beta\}}$ starts at $c_0\in C$ and ends at a unit $i_0$ outside the finite source set detected by $L$. This \emph{range escape} makes the F{\o}lner boundary vanish. An arbitrary subgroupoid need not contain the escaping arrow $\beta$. By contrast, an open invariant reduction contains every arrow whose source belongs to its unit space, which is exactly why Proposition~\ref{open-invariant-reduction} holds.
\end{remark}

\subsection{The one-unit permanence properties}

The group case remains fully hereditary, as it should.

\begin{corollary}\label{one-unit-permanence}
Within the class of countable one-unit groupoids, source-anchored left algebraic amenability is preserved under the following operations:
\begin{enumerate}
    \item passing to subgroupoids induced by subgroups;
    \item passing to quotient groups;
    \item taking extensions by amenable normal subgroups;
    \item taking finite direct products;
    \item taking directed unions of subgroups;
    \item taking countable direct sums of groups.
\end{enumerate}
\end{corollary}

\begin{proof}
By Theorem~\ref{group-case}, a countable one-unit groupoid is source-anchored left algebraically amenable over $\Bbbk$ exactly when its underlying group is amenable. Each assertion therefore follows from the corresponding classical permanence property for amenable groups recalled in Section~\ref{preliminaries}. For the direct-sum statement, every finitely generated subgroup of a countable direct sum is contained in a finite direct product of the coordinate groups; finite products and subgroups of amenable groups are amenable.
\end{proof}

\subsection{Disjoint unions}

The source-anchored condition behaves exactly as expected under topological disjoint unions.

\begin{proposition}\label{disjoint-unions}
Let $(\mathcal G_i)_{i\in I}$ be a non-empty family of Hausdorff ample groupoids, and let
$$
\mathcal G=\bigsqcup_{i\in I}\mathcal G_i
$$
be their topological disjoint union. Then $\mathcal G$ is algebraically amenable over $\Bbbk$ if and only if every $\mathcal G_i$ is algebraically amenable over $\Bbbk$.
\end{proposition}

\begin{proof}
Assume first that $\mathcal G$ is algebraically amenable. For each $i\in I$, the unit space $\mathcal G_i^{(0)}$ is an open invariant subset of $\mathcal G^{(0)}$, and
$$
\mathcal G|_{\mathcal G_i^{(0)}}=\mathcal G_i.
$$
Therefore Proposition~\ref{open-invariant-reduction} implies that every $\mathcal G_i$ is algebraically amenable.

Conversely, suppose that every $\mathcal G_i$ is algebraically amenable. Since a compact subset of a topological disjoint union meets only finitely many components, extension by zero yields an algebraic direct-sum decomposition
$$
A_{\Bbbk}(\mathcal G)=\bigoplus_{i\in I}A_{\Bbbk}(\mathcal G_i).
$$
Let $C\subseteq\mathcal G^{(0)}$ be non-empty, compact and open, let $L\subseteq A_{\Bbbk}(\mathcal G)$ be finite-dimensional and let $\varepsilon>0$.

Choose $i_0\in I$ such that
$$
C_{i_0}=C\cap\mathcal G_{i_0}^{(0)}
$$
is non-empty. The component $\mathcal G_{i_0}^{(0)}$ is clopen in $\mathcal G^{(0)}$, so $C_{i_0}$ is compact and open in $\mathcal G_{i_0}^{(0)}$. Let
$$
\pi_{i_0}:A_{\Bbbk}(\mathcal G)\longrightarrow A_{\Bbbk}(\mathcal G_{i_0})
$$
be the coordinate projection, and put
$$
L_{i_0}=\pi_{i_0}(L).
$$
Since $\mathcal G_{i_0}$ is algebraically amenable, there exists a non-zero finite-dimensional subspace
$$
W\subseteq A_{\Bbbk}(\mathcal G_{i_0})\mathbf 1_{C_{i_0}}
$$
such that
$$
\dim_{\Bbbk}(W+L_{i_0}W)\leq(1+\varepsilon)\dim_{\Bbbk}(W).
$$

We regard $W$ as a subspace of $A_{\Bbbk}(\mathcal G)$ by extension by zero. Since there are no composable arrows belonging to different components, one has
$$
aW=\pi_{i_0}(a)W
$$
for every $a\in A_{\Bbbk}(\mathcal G)$. Consequently,
$$
LW=L_{i_0}W.
$$
Moreover, $C_{i_0}\subseteq C$ implies
$$
W\subseteq A_{\Bbbk}(\mathcal G)\mathbf 1_C.
$$
Therefore
$$
\dim_{\Bbbk}(W+LW)=\dim_{\Bbbk}(W+L_{i_0}W)\leq(1+\varepsilon)\dim_{\Bbbk}(W).
$$
Thus $\mathcal G$ is algebraically amenable.
\end{proof}

\subsection{Products}

\begin{proposition}\label{products}
Let $\mathcal G$ and $\mathcal H$ be Hausdorff ample groupoids. If both are algebraically amenable, then $\mathcal G\times\mathcal H$ is algebraically amenable over $\Bbbk$.
\end{proposition}

\begin{proof}
The product $\mathcal G\times\mathcal H$ is again a Hausdorff ample groupoid. By the tensor-product theorem for Steinberg algebras, there is a natural $\Bbbk$-algebra isomorphism
$$
\Theta:A_{\Bbbk}(\mathcal G)\otimes_{\Bbbk}A_{\Bbbk}(\mathcal H)\longrightarrow A_{\Bbbk}(\mathcal G\times\mathcal H)
$$
satisfying
$$
\Theta(f\otimes g)(\gamma,\eta)=f(\gamma)g(\eta);
$$
see \cite[Theorem 4.3]{Rigby21}. In particular,
$$
\Theta(\mathbf 1_U\otimes\mathbf 1_V)=\mathbf 1_{U\times V}
$$
for compact open bisections $U\subseteq\mathcal G$ and $V\subseteq\mathcal H$. We use $\Theta$ to identify the two algebras.

Let
$$
C\subseteq(\mathcal G\times\mathcal H)^{(0)}=\mathcal G^{(0)}\times\mathcal H^{(0)}
$$
be non-empty, compact and open. Choose a point $(x,y)\in C$. Since both unit spaces have bases of compact open subsets, there exist non-empty compact open subsets
$$
C_{\mathcal G}\subseteq\mathcal G^{(0)}, \ C_{\mathcal H}\subseteq\mathcal H^{(0)}
$$
such that
$$
(x,y)\in C_{\mathcal G}\times C_{\mathcal H}\subseteq C.
$$

Let
$$
L\subseteq A_{\Bbbk}(\mathcal G)\otimes_{\Bbbk}A_{\Bbbk}(\mathcal H)
$$
be finite-dimensional and let $\varepsilon>0$. Choose a basis $\ell_1,\dots,\ell_n$ of $L$. Write each basis vector as a finite sum of elementary tensors,
$$
\ell_j=\sum_{r=1}^{m_j}a_{j,r}\otimes b_{j,r}.
$$
Set
$$
L_{\mathcal G}=\operatorname{span}_{\Bbbk}\{a_{j,r}\mid 1\leq j\leq n,\ 1\leq r\leq m_j\}
$$
and
$$
L_{\mathcal H}=\operatorname{span}_{\Bbbk}\{b_{j,r}\mid 1\leq j\leq n,\ 1\leq r\leq m_j\}.
$$
Then
$$
L\subseteq L_{\mathcal G}\otimes_{\Bbbk}L_{\mathcal H}.
$$
Choose
$$
\delta=\sqrt{1+\varepsilon}-1>0,
$$
so that
$$
(1+\delta)^2=1+\varepsilon.
$$

By algebraic amenability of $\mathcal G$, there exists a non-zero finite-dimensional subspace
$$
W_{\mathcal G}\subseteq A_{\Bbbk}(\mathcal G)\mathbf 1_{C_{\mathcal G}}
$$
such that
$$
\dim_{\Bbbk}\left(W_{\mathcal G}+L_{\mathcal G}W_{\mathcal G}\right)\leq(1+\delta)\dim_{\Bbbk}(W_{\mathcal G}).
$$
Similarly, there exists a non-zero finite-dimensional subspace
$$
W_{\mathcal H}\subseteq A_{\Bbbk}(\mathcal H)\mathbf 1_{C_{\mathcal H}}
$$
such that
$$
\dim_{\Bbbk}\left(W_{\mathcal H}+L_{\mathcal H}W_{\mathcal H}\right)\leq(1+\delta)\dim_{\Bbbk}(W_{\mathcal H}).
$$

Define
$$
W=W_{\mathcal G}\otimes_{\Bbbk}W_{\mathcal H}.
$$
Then $W$ is non-zero and finite-dimensional. Moreover,
$$
W\subseteq A_{\Bbbk}(\mathcal G\times\mathcal H) \mathbf 1_{C_{\mathcal G}\times C_{\mathcal H}}.
$$
Since $C_{\mathcal G}\times C_{\mathcal H}\subseteq C$, one has
$$
\mathbf 1_{C_{\mathcal G}\times C_{\mathcal H}}\ast\mathbf 1_C=\mathbf 1_{C_{\mathcal G}\times C_{\mathcal H}},
$$
and therefore
$$
W\subseteq A_{\Bbbk}(\mathcal G\times\mathcal H)\mathbf 1_C.
$$

Since $L\subseteq L_{\mathcal G}\otimes L_{\mathcal H}$,
$$
LW\subseteq(L_{\mathcal G}W_{\mathcal G})\otimes_{\Bbbk}(L_{\mathcal H}W_{\mathcal H}).
$$
Hence
$$
W+LW\subseteq\left(W_{\mathcal G}+L_{\mathcal G}W_{\mathcal G}\right)\otimes_{\Bbbk}\left(W_{\mathcal H}+L_{\mathcal H}W_{\mathcal H}\right).
$$
Taking dimensions gives
$$
\begin{aligned}
\dim_{\Bbbk}(W+LW) &\leq\dim_{\Bbbk}\left(W_{\mathcal G}+L_{\mathcal G}W_{\mathcal G}\right)\dim_{\Bbbk}\left(W_{\mathcal H}+L_{\mathcal H}W_{\mathcal H}\right)\\
&\leq(1+\delta)^2\dim_{\Bbbk}(W_{\mathcal G})\dim_{\Bbbk}(W_{\mathcal H})\\
&=(1+\varepsilon)\dim_{\Bbbk}(W).
\end{aligned}
$$
Thus $\mathcal G\times\mathcal H$ is algebraically amenable.
\end{proof}

\begin{corollary}\label{finite-products}
Every finite product of algebraically amenable Hausdorff ample groupoids is algebraically amenable.
\end{corollary}

\begin{proof}
Apply Proposition~\ref{products} inductively.
\end{proof}

\subsection{Directed unions}

\begin{proposition}\label{directed-unions}
Let $(\mathcal G_i)_{i\in I}$ be a directed family, under inclusion, of open Hausdorff ample subgroupoids of a Hausdorff ample groupoid $\mathcal G$. Suppose that
$$
\mathcal G=\bigcup_{i\in I}\mathcal G_i
$$
and that every $\mathcal G_i$ is algebraically amenable over $\Bbbk$. Then $\mathcal G$ is algebraically amenable over $\Bbbk$.
\end{proposition}

\begin{proof}
Let $C\subseteq\mathcal G^{(0)}$ be non-empty, compact and open, let $L\subseteq A_{\Bbbk}(\mathcal G)$ be finite-dimensional and let $\varepsilon>0$. Choose a basis $f_1,\dots,f_n$ of $L$, allowing the list to be empty if $L=\{0\}$. The set
$$
K=C\cup\bigcup_{j=1}^n\operatorname{supp}(f_j)
$$
is compact in $\mathcal G$.

The family $(\mathcal G_i)_{i\in I}$ is an open cover of $\mathcal G$, so finitely many members cover $K$. Directedness gives an index $i_0\in I$ containing all those finitely many subgroupoids. Hence
$$
K\subseteq\mathcal G_{i_0}.
$$
In particular,
$$
C\subseteq\mathcal G_{i_0}^{(0)} \text{ and } L\subseteq A_{\Bbbk}(\mathcal G_{i_0}),
$$
where the latter algebra is viewed inside $A_{\Bbbk}(\mathcal G)$ by extension by zero. Since $\mathcal G_{i_0}$ is open, the set $C$ is compact and open in $\mathcal G_{i_0}^{(0)}$.

Algebraic amenability of $\mathcal G_{i_0}$ yields a non-zero finite-dimensional subspace
$$
W\subseteq A_{\Bbbk}(\mathcal G_{i_0})\mathbf 1_C
$$
such that
$$
\dim_{\Bbbk}(W+LW)\leq(1+\varepsilon)\dim_{\Bbbk}(W).
$$
After extension by zero,
$$
W\subseteq A_{\Bbbk}(\mathcal G)\mathbf 1_C,
$$
and the same multiplication and the same dimension estimate hold in $A_{\Bbbk}(\mathcal G)$. Therefore $\mathcal G$ is algebraically amenable.
\end{proof}

\subsection{Finite-dimensional and AF-type approximations}

The following criterion is stronger and cleaner than assuming that the entire Steinberg algebra is finite-dimensional.

\begin{proposition}\label{locally-finite-dimensional-steinberg}
Let $\mathcal G$ be a Hausdorff ample groupoid. Suppose that every finite subset of $A_{\Bbbk}(\mathcal G)$ is contained in a finite-dimensional $\Bbbk$-subalgebra of $A_{\Bbbk}(\mathcal G)$. Then $\mathcal G$ is algebraically amenable over $\Bbbk$.
\end{proposition}

\begin{proof}
Let $C\subseteq\mathcal G^{(0)}$ be non-empty, compact and open, let $L\subseteq A_{\Bbbk}(\mathcal G)$ be finite-dimensional and let $\varepsilon>0$. Choose a basis $\ell_1,\dots,\ell_n$ of $L$. By hypothesis, there exists a finite-dimensional subalgebra
$$
B\subseteq A_{\Bbbk}(\mathcal G)
$$
containing
$$
\{\ell_1,\dots,\ell_n,\mathbf 1_C\}.
$$
Set
$$
W=B\mathbf 1_C.
$$
The space $W$ is finite-dimensional and
$$
\mathbf 1_C=\mathbf 1_C\ast\mathbf 1_C\in W,
$$
so $W\neq\{0\}$. Also,
$$
W\subseteq A_{\Bbbk}(\mathcal G)\mathbf 1_C.
$$
Since $L\subseteq B$ and $B$ is a subalgebra,
$$
LW=L(B\mathbf 1_C)\subseteq B\mathbf 1_C=W.
$$
Therefore
$$
W+LW=W
$$
and
$$
\dim_{\Bbbk}(W+LW)=\dim_{\Bbbk}(W)\leq(1+\varepsilon)\dim_{\Bbbk}(W).
$$
Thus $\mathcal G$ is algebraically amenable.
\end{proof}

\begin{corollary}\label{finite-dimensional-steinberg}
If $A_{\Bbbk}(\mathcal G)$ is finite-dimensional, then $\mathcal G$ is algebraically amenable over $\Bbbk$.
\end{corollary}

\begin{proof}
Apply Proposition~\ref{locally-finite-dimensional-steinberg} with the finite-dimensional subalgebra $A_{\Bbbk}(\mathcal G)$ itself.
\end{proof}

\begin{definition}[{\cite[\S 5.2]{ClarkHazrat20}}]\label{af-groupoid-definition}
Let $\mathcal G$ be a second-countable Hausdorff ample groupoid with compact unit space. An open subgroupoid $\mathcal H\subseteq\mathcal G$ is called \emph{elementary} if $\mathcal H$ is compact and principal and
$$
\mathcal H^{(0)}=\mathcal G^{(0)}.
$$
The groupoid $\mathcal G$ is called an \emph{AF groupoid} if there exists an increasing sequence
$$
\mathcal H_1\subseteq\mathcal H_2\subseteq\cdots
$$
of elementary subgroupoids such that
$$
\mathcal G=\bigcup_{n\geq 1}\mathcal H_n.
$$
\end{definition}

\begin{corollary}\label{af-groupoids}
Every AF groupoid is algebraically amenable over every field $\Bbbk$.
\end{corollary}

\begin{proof}
For an AF groupoid, the Steinberg algebra over a field is ultramatricial; see \cite[\S 5]{ClarkHazrat20}. In particular, it is a directed union of finite-dimensional subalgebras, so every finite subset is contained in a finite-dimensional subalgebra. Proposition~\ref{locally-finite-dimensional-steinberg} gives the conclusion.
\end{proof}

\subsection{Transitive discrete groupoids}\label{transitive-discrete-groupoids}

The behavior of a transitive discrete groupoid depends sharply on the size of its unit space. If the unit space is infinite, a F{\o}lner subspace may escape from any prescribed finite collection of source fibers. For a finite unit space this phenomenon disappears, and algebraic amenability is determined exactly by the isotropy group.

\begin{lemma}\label{infinite-transitive-discrete-groupoid}
Let $\mathcal H$ be a countable transitive discrete groupoid whose unit space $\mathcal H^{(0)}$ is infinite. Then $\mathcal H$ is source-anchored left algebraically amenable over every field $\Bbbk$.
\end{lemma}

\begin{proof}
Let $C\subseteq\mathcal H^{(0)}$ be non-empty, compact and open, let $L\subseteq A_{\Bbbk}(\mathcal H)$ be finite-dimensional and let $\varepsilon>0$. Since $\mathcal H^{(0)}$ is discrete, the compact set $C$ is finite.

Choose a basis $f_1,\dots,f_m$ of $L$ and put
$$
S=\bigcup_{j=1}^m\operatorname{supp}(f_j).
$$
Every $f_j$ has finite support because $\mathcal H$ is discrete and $f_j$ is compactly supported. Hence $S$ and $s(S)$ are finite. Since $\mathcal H^{(0)}$ is infinite, there exists
$$
v\in\mathcal H^{(0)}\setminus s(S).
$$
Choose $c\in C$. By transitivity, there is an arrow $\beta\in\mathcal H$ such that
$$
s(\beta)=c,\ r(\beta)=v.
$$
Define
$$
W=\operatorname{span}_{\Bbbk}\{\mathbf 1_{\{\beta\}}\}.
$$
Since $s(\beta)=c\in C$, Proposition~\ref{unit-space-cut-downs} gives
$$
W\subseteq A_{\Bbbk}(\mathcal H)\mathbf 1_C.
$$

We claim that $LW=\{0\}$. Indeed, let $f\in L$. Every arrow in $\operatorname{supp}(f)$ belongs to $S$. A product $\alpha\beta$ can be defined only if
$$
s(\alpha)=r(\beta)=v.
$$
However, $v\notin s(S)$, so no arrow $\alpha\in\operatorname{supp}(f)$ is composable on the left with $\beta$. Consequently,
$$
f\ast\mathbf 1_{\{\beta\}}=0.
$$
Since this holds for every $f\in L$, we obtain $LW=\{0\}$. Therefore
$$
W+LW=W,
$$
and hence
$$
\dim_{\Bbbk}(W+LW)=\dim_{\Bbbk}(W)\leq(1+\varepsilon)\dim_{\Bbbk}(W).
$$
Thus $\mathcal H$ is source-anchored left algebraically amenable over $\Bbbk$. Since the argument does not depend on the field, the conclusion holds over every field.
\end{proof}

The next standard coordinatization makes the finite-unit case completely explicit.

\begin{lemma}[Matrix model for a transitive groupoid; cf. {\cite[Example 3.2 and Proposition 4.3]{ClarkHazrat20}}]\label{finite-transitive-matrix-model}
Let $\mathcal H$ be a transitive discrete groupoid with finite unit space
$$
O=\mathcal H^{(0)}=\{u_1,\dots,u_n\}.
$$
Fix $u=u_1$ and put
$$
K=\mathcal H_u^u.
$$
For each $1\leq i\leq n$, choose an arrow $\tau_i\in\mathcal H$ such that
$$
s(\tau_i)=u,\ r(\tau_i)=u_i,
$$
and take $\tau_1=u$. Then the map
$$
\Psi:O\times K\times O\longrightarrow\mathcal H, \ \Psi(u_i,k,u_j)=\tau_i k\tau_j^{-1},
$$
is a groupoid isomorphism. Moreover, for every field $\Bbbk$, it induces a $\Bbbk$-algebra isomorphism
$$
\Phi:A_{\Bbbk}(\mathcal H)\longrightarrow M_n(\Bbbk K)
$$
determined by
$$
\Phi\left(\mathbf 1_{\{\tau_i k\tau_j^{-1}\}}\right)=e_{ij}(k),
$$
where $e_{ij}(k)$ denotes the matrix whose $(i,j)$-entry is $k$ and whose other entries are zero. In particular,
$$
\Phi\left(\mathbf 1_{\{u_i\}}\right)=e_{ii}(1_{\Bbbk K})
$$
and
$$
\Phi\left(A_{\Bbbk}(\mathcal H)\mathbf 1_{\{u_j\}}\right)=M_n(\Bbbk K)e_{jj}(1_{\Bbbk K}).
$$
\end{lemma}

\begin{proof}
Let $\gamma\in\mathcal H$ satisfy
$$
s(\gamma)=u_j,\ r(\gamma)=u_i.
$$
Then
$$
k_{\gamma}=\tau_i^{-1}\gamma\tau_j
$$
is an arrow from $u$ to $u$, and therefore $k_{\gamma}\in K$. Moreover,
$$
\gamma=\tau_i k_{\gamma}\tau_j^{-1}.
$$
This expression is unique: if
$$
\tau_i k\tau_j^{-1}=\tau_i \ell\tau_j^{-1},
$$
then multiplication on the left by $\tau_i^{-1}$ and on the right by $\tau_j$ gives $k=\ell$. Hence $\Psi$ is bijective.

The element $(u_i,k,u_j)$ has source $u_j$ and range $u_i$. If $u_j=u_p$, then
$$
\begin{aligned}
\Psi(u_i,k,u_j)\Psi(u_p,\ell,u_q)&=
\tau_i k\tau_j^{-1}\tau_j\ell\tau_q^{-1} \\
&= \tau_i k\ell\tau_q^{-1} \\
&= \Psi(u_i,k\ell,u_q).
\end{aligned}
$$
Also,
$$
\Psi(u_i,k,u_j)^{-1}= \tau_j k^{-1}\tau_i^{-1} = \Psi(u_j,k^{-1},u_i).
$$
Thus $\Psi$ preserves source, range, multiplication, inversion and units. Since both groupoids are discrete, it is an isomorphism of topological groupoids.

The Steinberg algebra of a discrete groupoid is the algebra of finitely supported $\Bbbk$-valued functions on its arrow space. The singleton characteristic functions form a $\Bbbk$-basis. For $a,b\in\Bbbk K$, let $e_{ij}(a)$ denote the matrix with $a$ in the $(i,j)$-entry and zero elsewhere. Then
$$
e_{ij}(a)e_{pq}(b)=\begin{cases}
 e_{iq}(ab), & j=p,\\
 0, & j\neq p.
\end{cases}
$$
On the groupoid side,
$$
\mathbf 1_{\{\tau_i k\tau_j^{-1}\}}\ast\mathbf 1_{\{\tau_p\ell\tau_q^{-1}\}}=\begin{cases}
\mathbf 1_{\{\tau_i k\ell\tau_q^{-1}\}}, & j=p,\\
0, & j\neq p.
\end{cases}
$$
Therefore the displayed assignment extends uniquely to a multiplicative $\Bbbk$-linear bijection
$$
\Phi:A_{\Bbbk}(\mathcal H)\longrightarrow M_n(\Bbbk K).
$$
Finally, the unit $u_i$ corresponds to the arrow $\tau_i1_K\tau_i^{-1}$, so
$$
\Phi(\mathbf 1_{\{u_i\}})=e_{ii}(1_{\Bbbk K}).
$$
Since $\Phi$ is multiplicative, it maps the left ideal generated by $\mathbf 1_{\{u_j\}}$ onto the left ideal generated by $e_{jj}(1_{\Bbbk K})$. This proves the final assertion.
\end{proof}

\begin{theorem}\label{finite-transitive-isotropy}
Let $\mathcal H$ be a countable transitive discrete groupoid with finite non-empty unit space, let $u\in\mathcal H^{(0)}$, and put
$$
K=\mathcal H_u^u.
$$
Then the following conditions are equivalent:
\begin{enumerate}
    \item $\mathcal H$ is source-anchored left algebraically amenable over $\Bbbk$;
    \item the isotropy group $K$ is amenable;
    \item every isotropy group of $\mathcal H$ is amenable.
\end{enumerate}
Moreover, if $\mathcal H$ is source-anchored left algebraically amenable over one field, then it is source-anchored left algebraically amenable over every field.
\end{theorem}

\begin{proof}
Write
$$
\mathcal H^{(0)}=\{u_1,\dots,u_n\},\ u_1=u,
$$
and use Lemma~\ref{finite-transitive-matrix-model} to identify
$$
A_{\Bbbk}(\mathcal H)\cong M_n(R), \ R=\Bbbk K.
$$
Under this identification, $\mathbf 1_{\{u_i\}}$ corresponds to $e_{ii}(1_R)$.

We first prove $(2)\Rightarrow(1)$. Assume that $K$ is amenable. By Theorem~\ref{group-case}, the group algebra $R=\Bbbk K$ is left algebraically amenable.

Let $C\subseteq\mathcal H^{(0)}$ be non-empty, let
$$
L\subseteq M_n(R)
$$
be finite-dimensional and let $\varepsilon>0$. Since the unit space is finite and discrete, $C$ is compact and open. Choose $u_\ell\in C$. Choose a basis
$$
A^{(1)},\dots,A^{(m)}
$$
of $L$, and let $E\subseteq R$ be the finite-dimensional subspace spanned by $1_R$ and all matrix entries of the matrices $A^{(1)},\dots,A^{(m)}$. Then every entry of every matrix in $L$ belongs to $E$.

Since $R$ is left algebraically amenable, there exists a non-zero finite-dimensional subspace $V\subseteq R$ such that
$$
\dim_{\Bbbk}(V+EV)\leq(1+\varepsilon)\dim_{\Bbbk}(V).
$$
Because $1_R\in E$, one has $V\subseteq EV$, and consequently
$$
\dim_{\Bbbk}(EV)\leq(1+\varepsilon)\dim_{\Bbbk}(V).
$$

Define
$$
W=\bigoplus_{i=1}^n e_{i\ell}(V),
$$
where
$$
e_{i\ell}(V)=\{e_{i\ell}(v)\mid v\in V\}.
$$
Thus $W$ consists of the matrices supported in the $\ell$-th column whose entries belong to $V$. It is non-zero and finite-dimensional, and
$$
\dim_{\Bbbk}(W)=n\dim_{\Bbbk}(V).
$$
Moreover,
$$
W\subseteq M_n(R)e_{\ell\ell}(1_R)\subseteq M_n(R)\sum_{u_i\in C}e_{ii}(1_R).
$$
Hence, under the inverse of $\Phi$, the space $W$ is contained in
$$
A_{\Bbbk}(\mathcal H)\mathbf 1_C.
$$

Let $A=(a_{ij})\in L$ and let
$$
w=\sum_{j=1}^n e_{j\ell}(v_j)\in W.
$$
The product $Aw$ is supported in the $\ell$-th column, and its $(i,\ell)$-entry is
$$
\sum_{j=1}^n a_{ij}v_j.
$$
Since $a_{ij}\in E$ and $v_j\in V$, this entry belongs to $EV$. Therefore
$$
LW\subseteq\bigoplus_{i=1}^n e_{i\ell}(EV).
$$
Because $V\subseteq EV$, the same ambient space also contains $W$. It follows that
$$
\begin{aligned}
\dim_{\Bbbk}(W+LW)
&\leq
n\dim_{\Bbbk}(EV)
\\
&\leq
(1+\varepsilon)n\dim_{\Bbbk}(V)
\\
&=
(1+\varepsilon)\dim_{\Bbbk}(W).
\end{aligned}
$$
Thus $\mathcal H$ is source-anchored left algebraically amenable over $\Bbbk$.

We now prove $(1)\Rightarrow(2)$. Assume that $\mathcal H$ is source-anchored left algebraically amenable over $\Bbbk$. Let $E_0\subseteq R$ be finite-dimensional and let $\varepsilon>0$. Put
$$
E=E_0+\Bbbk 1_R
$$
and define the finite-dimensional subspace
$$
M_n(E)=\{(a_{ij})\in M_n(R)\mid a_{ij}\in E\}.
$$
Apply source-anchored algebraic amenability to the compact open singleton $\{u_1\}$ and to the multiplier space $M_n(E)$. We obtain a non-zero finite-dimensional subspace
$$
W\subseteq M_n(R)e_{11}(1_R)
$$
such that
$$\dim_{\Bbbk}\left(W+M_n(E)W\right)\leq(1+\varepsilon)\dim_{\Bbbk}(W).
$$

Identify the left ideal $M_n(R)e_{11}(1_R)$ with the column space $R^n$. For $1\leq i\leq n$, let
$$
\pi_i:R^n\longrightarrow R
$$
be the $i$-th coordinate projection, and put
$$
V=\sum_{i=1}^n\pi_i(W).
$$
The space $V$ is finite-dimensional. It is non-zero because $W\neq\{0\}$.

We claim that
$$
M_n(E)W=(EV)^n.
$$
For the inclusion from left to right, let $A=(a_{ij})\in M_n(E)$ and $w=(w_1,\dots,w_n)^{\mathsf T}\in W$. The $i$-th coordinate of $Aw$ is
$$
\sum_{j=1}^n a_{ij}w_j,
$$
which belongs to $EV$ because $a_{ij}\in E$ and $w_j\in\pi_j(W)\subseteq V$. Thus
$$
M_n(E)W\subseteq(EV)^n.
$$

For the reverse inclusion, let $a\in E$, let $w\in W$, and fix indices $i$ and $j$. Multiplication by the elementary matrix $e_{ij}(a)\in M_n(E)$ produces a column vector whose $i$-th coordinate is $a\pi_j(w)$ and whose other coordinates are zero. Since
$$
V=\sum_{j=1}^n\pi_j(W),
$$
the vectors $a\pi_j(w)$ span $EV$ when $a$, $j$ and $w$ vary. Therefore every column vector having one prescribed coordinate in $EV$ and all other coordinates equal to zero belongs to $M_n(E)W$. Summing over the coordinates gives
$$
(EV)^n\subseteq M_n(E)W.
$$
This proves the claim.

Since $1_R\in E$, we have $V\subseteq EV$. Moreover,
$$
W\subseteq V^n\subseteq(EV)^n=M_n(E)W.
$$
Hence
$$
W+M_n(E)W=M_n(E)W=(EV)^n.
$$
It follows that
$$
\begin{aligned}
n\dim_{\Bbbk}(EV)&=\dim_{\Bbbk}\left(W+M_n(E)W\right)\\
&\leq(1+\varepsilon)\dim_{\Bbbk}(W) \\
&\leq(1+\varepsilon)n\dim_{\Bbbk}(V).
\end{aligned}
$$
After cancelling $n$, we obtain
$$
\dim_{\Bbbk}(EV)\leq(1+\varepsilon)\dim_{\Bbbk}(V).
$$
Finally, since $1_R\in E$ and $E_0\subseteq E$,
$$
V+E_0V\subseteq EV.
$$
Therefore
$$
\dim_{\Bbbk}(V+E_0V)\leq(1+\varepsilon)\dim_{\Bbbk}(V).
$$
Thus $R=\Bbbk K$ is left algebraically amenable. By Theorem~\ref{group-case}, the group $K$ is amenable. This proves $(1)\Rightarrow(2)$.

For every $x\in\mathcal H^{(0)}$, choose an arrow $\gamma\in\mathcal H$ with
$$
s(\gamma)=u,\ r(\gamma)=x.
$$
Conjugation by $\gamma$ gives a group isomorphism
$$
K\longrightarrow\mathcal H_x^x, \ k\longmapsto\gamma k\gamma^{-1}.
$$
Hence $(2)$ and $(3)$ are equivalent.

Finally, condition~$(2)$ is independent of the coefficient field. If condition~$(1)$ holds over one field, then $(2)$ holds; applying $(2)\Rightarrow(1)$ over any other field proves the final assertion.
\end{proof}

\begin{corollary}\label{finite-orbit-isotropy}
Let $\mathcal G$ be a countable discrete groupoid, let $O\subseteq\mathcal G^{(0)}$ be a finite orbit, and let $u\in O$. Then the reduction $\mathcal G|_O$ is source-anchored left algebraically amenable over $\Bbbk$ if and only if the isotropy group $\mathcal G_u^u$ is amenable.
\end{corollary}

\begin{proof}
The reduction $\mathcal G|_O$ is a transitive discrete groupoid with finite unit space $O$, and its isotropy group at $u$ is $\mathcal G_u^u$. The conclusion follows from Theorem~\ref{finite-transitive-isotropy}.
\end{proof}

\begin{remark}\label{finite-versus-infinite-orbits}
Lemma~\ref{infinite-transitive-discrete-groupoid} and Theorem~\ref{finite-transitive-isotropy} exhibit a sharp dichotomy. On an infinite orbit, source anchoring still permits the F{\o}lner subspace to escape through a range unit not seen by the prescribed finite-dimensional multiplier space, so the isotropy may be non-amenable. On a finite orbit, no such escape is possible: the finite matrix amplification forces the F{\o}lner estimate to descend to the isotropy group algebra.
\end{remark}

\subsection{Hereditary subgroupoid amenability}

Proposition~\ref{arbitrary-subgroupoids-fail} motivates recording separately the genuinely hereditary form of the property.

\begin{definition}\label{hereditarily-algebraically-amenable}
A Hausdorff ample groupoid $\mathcal G$ is called \emph{hereditarily source-anchored left algebraically amenable over $\Bbbk$} if every non-empty Hausdorff ample subgroupoid $\mathcal H\subseteq\mathcal G$, endowed with the subspace topology, is source-anchored left algebraically amenable over $\Bbbk$.
\end{definition}

The hereditary property admits a complete isotropy characterization for countable discrete groupoids.

\begin{theorem}\label{hereditary-discrete-characterization}
Let $\mathcal G$ be a countable discrete groupoid. The following conditions are equivalent:
\begin{enumerate}
    \item $\mathcal G$ is hereditarily source-anchored left algebraically amenable over $\Bbbk$;
    \item every isotropy group $\mathcal G_u^u$, with $u\in\mathcal G^{(0)}$, is amenable.
\end{enumerate}
Moreover, if these conditions hold for one field, then they hold for every field.
\end{theorem}

\begin{proof}
Assume first that $\mathcal G$ is hereditarily source-anchored left algebraically amenable. Fix $u\in\mathcal G^{(0)}$. The isotropy group $\mathcal G_u^u$, viewed as a one-unit groupoid, is a discrete Hausdorff ample subgroupoid of $\mathcal G$. Hence it is source-anchored left algebraically amenable. By Theorem~\ref{group-case}, the group $\mathcal G_u^u$ is amenable. This proves $(1)\Rightarrow(2)$.

Conversely, suppose that every isotropy group of $\mathcal G$ is amenable. Let $\mathcal H\subseteq\mathcal G$ be a non-empty subgroupoid. Since $\mathcal G$ is discrete, $\mathcal H$ is also a countable discrete Hausdorff ample groupoid. For every $u\in\mathcal H^{(0)}$,
$$
\mathcal H_u^u\leq\mathcal G_u^u.
$$
Subgroups of amenable groups are amenable, so every isotropy group of $\mathcal H$ is amenable.

Decompose $\mathcal H$ into its orbit subgroupoids:
$$
\mathcal H=\bigsqcup_{O\in\mathcal H^{(0)}/\mathcal H}\mathcal H|_O.
$$
Because $\mathcal H$ is discrete, each orbit subgroupoid is open, and the displayed decomposition is a topological disjoint union. Each $\mathcal H|_O$ is transitive. If $O$ is infinite, Lemma~\ref{infinite-transitive-discrete-groupoid} shows that $\mathcal H|_O$ is algebraically amenable. If $O$ is finite, its isotropy groups are amenable, and Theorem~\ref{finite-transitive-isotropy} gives the same conclusion. Proposition~\ref{disjoint-unions} now implies that $\mathcal H$ is algebraically amenable. Since $\mathcal H$ was arbitrary, $\mathcal G$ is hereditarily source-anchored left algebraically amenable. This proves $(2)\Rightarrow(1)$.

Condition $(2)$ is independent of the coefficient field. Hence the equivalence shows that the hereditary property is also independent of the field in the countable discrete setting.
\end{proof}

\begin{corollary}\label{hereditary-discrete-classes}
Each of the following countable discrete groupoids is hereditarily source-anchored left algebraically amenable over every field:
\begin{enumerate}
    \item every principal groupoid;
    \item every groupoid whose isotropy groups are finite;
    \item every groupoid whose isotropy groups are abelian, locally finite or solvable;
    \item every topological disjoint union of amenable groups.
\end{enumerate}
\end{corollary}

\begin{proof}
A principal groupoid has trivial isotropy groups. Finite, abelian, locally finite and solvable groups are amenable, as recalled in Section~\ref{preliminaries}. A disjoint union of amenable groups has amenable isotropy at every unit. Theorem~\ref{hereditary-discrete-characterization} applies in each case.
\end{proof}

\begin{remark}\label{subgroupoid-summary}
The subgroupoid picture can now be summarized as follows.
\begin{enumerate}
    \item Open invariant reductions, equivalently open source-saturated subgroupoids, inherit algebraic amenability from the ambient groupoid by Proposition~\ref{open-invariant-reduction}.
    \item Arbitrary open subgroupoids need not inherit the property, by Proposition~\ref{arbitrary-subgroupoids-fail}.
    \item Requiring hereditary algebraic amenability restores the subgroup property by definition, and Theorem~\ref{hereditary-discrete-characterization} gives an exact criterion in the countable discrete case: all isotropy groups must be amenable.
\end{enumerate}
Thus the classical subgroup property is recovered in the one-unit case and in the hereditary discrete class described above. These are positive cases, not an exhaustive classification of subgroupoid inclusions for which inheritance holds.
\end{remark}

\section{Finite-index extensions of transformation groupoids}\label{finite-index-transformation-groupoids}

We now consider global actions on compact zero-dimensional spaces. In this setting, compact open bisections retain the labels of the acting group, and these labels provide both a F{\o}lner construction and an obstruction to algebraic amenability. We first show that the source-anchored and unanchored conditions are equivalent to amenability of the acting group. We then give a direct induction procedure for finite-index subgroups, with an explicit estimate that preserves the relative algebraic boundary.

Throughout this section, $G$ is a countable discrete group acting by homeomorphisms on a non-empty compact Hausdorff space $X$ with a basis of clopen subsets. The coefficient field $\Bbbk$ is arbitrary. As in Section~\ref{algebraic-folner}, algebraic amenability of a groupoid refers to the source-anchored left-handed condition unless explicitly stated otherwise.

\subsection{Transformation groupoids and their homogeneous pieces}

\begin{definition}[{\cite[Example 3.3]{ClarkHazrat20}; cf. \cite[Example 4.10]{Ste09}}]
\label{transformation-groupoid-definition}
The \emph{transformation groupoid} $X\rtimes G$ is the set $X\times G$ with range and source maps
$$
r(x,g)=x,\ s(x,g)=g^{-1}x,
$$
composition $(x,g)(g^{-1}x,h)=(x,gh)$, and inversion
$$
(x,g)^{-1}=(g^{-1}x,g^{-1}).
$$
Its unit space is identified with $X$ through $x\mapsto(x,e_G)$, and its topology is the product topology. We use the range-first convention for its arrows.
\end{definition}

For a subgroup $Q\leq G$, write
$$
A_Q=A_{\Bbbk}(X\rtimes Q),\ D=C^{\mathrm{lc}}(X,\Bbbk).
$$
Here $D$ is the algebra of locally constant functions on $X$, identified with the functions supported on the unit space. Compactness makes their supports compact. We identify $A_Q$ with its extension-by-zero image in $A_G$. This is an algebra embedding because $X\rtimes Q$ is an open subgroupoid and products of arrows whose labels lie in $Q$ again have labels in $Q$.

\begin{lemma}\label{transformation-homogeneous-model}
The groupoid $X\rtimes G$ is Hausdorff and ample. Its rectangular compact open bisections satisfy
$$
(D_1\times\{g\})(D_2\times\{h\})=(D_1\cap gD_2)\times\{gh\}
$$
for clopen $D_1,D_2\subseteq X$. Every compact open bisection $U$ has a finite disjoint decomposition
$$
U=\bigsqcup_{g\in K_U}D_{U,g}\times\{g\},
$$
where $K_U\subseteq G$ is finite and each $D_{U,g}$ is clopen.

For $g\in G$, set $u_g=\mathbf 1_{X\times\{g\}}$. Then $u_g$ is invertible in $A_G$, with inverse $u_{g^{-1}}$, and
$$
u_g u_h=u_{gh},\ u_g f=(f\circ g^{-1})u_g\quad(f\in D).
$$
Every element of $A_G$ has a unique finite expression $\sum_g f_g u_g$, with $f_g\in D$, and also a unique finite expression $\sum_g u_g d_g$, with $d_g\in D$. In particular,
$$
A_G=\bigoplus_{g\in G}D u_g
=\bigoplus_{g\in G}u_g D.
$$
All products in these formulas are convolution products.
\end{lemma}

\begin{proof}
On $D_1\times\{g\}$, the range and source maps are homeomorphisms onto $D_1$ and $g^{-1}D_1$, respectively. Such sets form a compact open basis, and $X\times G$ is Hausdorff. The multiplication formula follows by checking composability. The projection of a compact open bisection $U$ onto the discrete space $G$ is finite; its slices $D_{U,g}$ are open and compact in $X$, hence clopen. This proves the decomposition.

The convolution identities follow from the rectangular multiplication formula. In particular, $u_{e_G}=\mathbf 1_X$ is the identity and $u_g u_{g^{-1}}=u_{g^{-1}}u_g=\mathbf 1_X$. If $a\in A_G$, compactness of its support implies that only finitely many slices $X\times\{g\}$ meet that support. Its restriction to each such slice is a function $f_g\in D$, giving $a=\sum_g f_g u_g$. The expression is unique because distinct slices are disjoint. Since
$$
f_g u_g=u_g(f_g\circ g),
$$
the right-coefficient expression is also unique.
\end{proof}

In particular, $\delta_g\mapsto u_g$ defines a unital embedding $\Bbbk G\hookrightarrow A_G$. It is injective even when the action on $X$ is not faithful: distinct group elements still label disjoint slices of the transformation groupoid. This observation concerns transformation groupoids, not quotients obtained by identifying germs.

\subsection{The compact-action criterion}

\begin{proposition}\label{amenable-compact-transformation}
If $G$ is amenable, then $X\rtimes G$ is source-anchored left algebraically amenable over every field.
\end{proposition}

\begin{proof}
Fix a non-empty clopen set $C\subseteq X$, a finite-dimensional space $L\subseteq A_G$, and $\varepsilon>0$. We use Proposition~\ref{anchored-finite-dimensional-test}. If $L=\{0\}$, the space $\Bbbk\mathbf 1_C$ suffices. Otherwise, choose a basis $a_1,\dots,a_q$ of $L$ and a finite set $K\subseteq G$ containing $e_G$ such that
$$
a_i=\sum_{k\in K}f_{i,k}u_k,\ f_{i,k}\in D.
$$

The group F{\o}lner condition gives a non-empty finite set $F\subseteq G$ with
$$
|KF|\leq(1+\varepsilon)|F|.
$$
Indeed, choose $F$ so that $|kF\setminus F|<\varepsilon|F|/|K|$ for every $k\in K$. Since $e_G\in K$, the union bound gives the displayed estimate. This version follows from the condition in the preliminaries by applying it to $K^{-1}$ and using $|kF\setminus F|=|F\setminus k^{-1}F|$.

The finitely many functions
$$
(f_{i,k}\circ h)|_C,\ 1\leq i\leq q,\ k\in K,\ h\in KF,
$$
are locally constant on the compact space $C$. Each has finite image, and their fibers are clopen. The non-empty intersections of these fibers therefore form a finite clopen partition $\mathcal{Q}$ of $C$ on whose members all these functions are constant. Its members are also clopen in $X$.

For a finite set $S\subseteq G$, put
$$
W(S,\mathcal{Q})=\operatorname{span}_{\Bbbk}\{u_s\mathbf 1_P\mid s\in S,\ P\in\mathcal{Q}\}.
$$
Since $u_s\mathbf 1_P=\mathbf 1_{sP\times\{s\}}$, these functions have pairwise disjoint non-empty supports. Consequently,
$$
\dim_{\Bbbk}W(S,\mathcal{Q})=|S|\,|\mathcal{Q}|.
$$
Moreover, their sources lie in $P\subseteq C$, so $W(S,\mathcal{Q})\subseteq A_G\mathbf 1_C$.

Take $W=W(F,\mathcal{Q})$. For $f\in F$ and $P\in\mathcal{Q}$, we have
$$
a_i(u_f\mathbf 1_P)=\sum_{k\in K}f_{i,k}u_{kf}\mathbf 1_P.
$$
The function $f_{i,k}$ is constant on $kfP$, by the choice of $\mathcal{Q}$. Writing this constant as $\lambda_{i,k,f,P}$, we obtain
$$
f_{i,k}u_{kf}\mathbf 1_P=\lambda_{i,k,f,P}u_{kf}\mathbf 1_P.
$$
Thus $LW\subseteq W(KF,\mathcal{Q})$. Since $F\subseteq KF$, also $W\subseteq W(KF,\mathcal{Q})$, and hence
$$
\dim_{\Bbbk}(W+LW)\leq |KF|\,|\mathcal{Q}|\leq(1+\varepsilon)|F|\,|\mathcal{Q}|=(1+\varepsilon)\dim_{\Bbbk}(W).
$$
The space $W$ is non-zero and has the prescribed source anchor.
\end{proof}

To prove the converse, we need a finite-dimensional observation about regular modules. It is a module version of the F{\o}lner condition; compare \cite[Definition 1.3]{Bartholdi08}. The argument is included because algebraic amenability does not, in general, pass to arbitrary subalgebras.

\begin{lemma}\label{finite-regular-copies-folner}
Let $R$ be a non-zero unital $\Bbbk$-algebra, let $E\subseteq R$ be finite-dimensional with $1_R\in E$, let $m\geq1$, and let $\varepsilon>0$. Let $R$ act coordinatewise on $R^m$. If a non-zero finite-dimensional subspace $W\subseteq R^m$ satisfies
$$
\dim_{\Bbbk}(EW)\leq(1+\varepsilon)\dim_{\Bbbk}(W),
$$
then there is a non-zero finite-dimensional subspace $V\subseteq R$ such that
$$
\dim_{\Bbbk}(EV)\leq(1+\varepsilon)\dim_{\Bbbk}(V).
$$
\end{lemma}

\begin{proof}
Put $U=EW$, and let
$$
F_j=R^j\oplus\{0\}^{m-j}\subseteq R^m, \ 0\leq j\leq m.
$$
For $1\leq j\leq m$, let $\pi_j$ be the $j$th coordinate projection and define
$$
V_j=\pi_j(W\cap F_j),\ Z_j=\pi_j(U\cap F_j).
$$
The kernel of $\pi_j$ on $W\cap F_j$ is $W\cap F_{j-1}$. The corresponding assertion holds for $U$. Rank-nullity and telescoping therefore give
$$
\dim_{\Bbbk}(W)=\sum_{j=1}^m\dim_{\Bbbk}(V_j),\ \dim_{\Bbbk}(U)=\sum_{j=1}^m\dim_{\Bbbk}(Z_j).
$$

For $a\in E$ and $w\in W\cap F_j$, the vector $aw$ belongs to $U\cap F_j$ and has $j$th coordinate $a\pi_j(w)$. It follows that $EV_j\subseteq Z_j$. Thus
$$
\sum_{j=1}^m\dim_{\Bbbk}(EV_j)\leq\dim_{\Bbbk}(U)\leq(1+\varepsilon)\sum_{j=1}^m\dim_{\Bbbk}(V_j).
$$
At least one $V_j$ is non-zero because $W\neq\{0\}$. If every non-zero $V_j$ violated the desired inequality, summing those strict violations would contradict the last display. Choosing an index for which it holds and setting $V=V_j$ proves the lemma.
\end{proof}

\begin{theorem}\label{compact-transformation-criterion}
Under the standing hypotheses on $G\curvearrowright X$, the following are equivalent:
\begin{enumerate}
\item $G$ is amenable;
\item $A_{\Bbbk}(X\rtimes G)$ is left algebraically amenable;
\item $X\rtimes G$ is weakly left bisectionally algebraically amenable;
\item $X\rtimes G$ is source-anchored left algebraically amenable.
\end{enumerate}
Consequently, these properties are independent of the coefficient field for global actions on non-empty compact zero-dimensional spaces.
\end{theorem}

\begin{proof}
Proposition~\ref{amenable-compact-transformation} gives $(1)\Rightarrow(4)$, Proposition~\ref{anchored-implies-weak} gives $(4)\Rightarrow(3)$, and Proposition~\ref{weak-bisectional-equivalence} gives $(3)\Leftrightarrow(2)$.

Assume $(2)$ and identify $R=\Bbbk G$ with the unital subalgebra spanned by the $u_g$. Let $E\subseteq R$ be any finite-dimensional subspace containing $1_R$, and let $\varepsilon>0$. Algebraic amenability of $A_G$ provides a non-zero finite-dimensional subspace $W\subseteq A_G$ such that
$$
\dim_{\Bbbk}(EW)=\dim_{\Bbbk}(W+EW)\leq(1+\varepsilon)\dim_{\Bbbk}(W).
$$

Expand a basis of $W$ in the right-coefficient form $\sum_g u_g d_g$ from Lemma~\ref{transformation-homogeneous-model}. There are only finitely many coefficient functions in these expansions. Choose a finite clopen partition $\mathcal{Q}$ of $X$ on whose members every one of these functions is constant. Then
$$
W\subseteq M_{\mathcal{Q}}:=\bigoplus_{P\in\mathcal{Q}}R\mathbf 1_P.
$$
This is a free left $R$-module of rank $|\mathcal{Q}|$. To verify freeness, expand a finite relation $\sum_P r_P\mathbf 1_P=0$ using $r_P=\sum_g\lambda_{g,P}u_g$. The functions $u_g\mathbf 1_P=\mathbf 1_{gP\times\{g\}}$ have pairwise disjoint non-empty supports as $(g,P)$ varies, so every $\lambda_{g,P}$ is zero. Furthermore,
$$
u_h(u_g\mathbf 1_P)=u_{hg}\mathbf 1_P,
$$
so the identification $M_{\mathcal{Q}}\cong R^{|\mathcal{Q}|}$ intertwines left multiplication by $R$ with coordinatewise multiplication. In particular, $EW\subseteq M_{\mathcal{Q}}$.

Lemma~\ref{finite-regular-copies-folner} now gives a non-zero finite-dimensional $V\subseteq R$ with
$$
\dim_{\Bbbk}(V+EV)=\dim_{\Bbbk}(EV)\leq(1+\varepsilon)\dim_{\Bbbk}(V).
$$
Testing subspaces $E$ containing $1_R$ suffices, since any finite-dimensional subspace can be enlarged by $\Bbbk1_R$. Hence $R=\Bbbk G$ is left algebraically amenable. Bartholdi's criterion, recalled in Theorem~\ref{bartholdi-group-algebra}, implies that $G$ is amenable; see \cite[Corollary 1.5]{Bartholdi08}. This proves $(2)\Rightarrow(1)$.

The equivalence with group amenability holds for every field, which proves the final assertion.
\end{proof}

\begin{remark}\label{compact-action-scope}
The theorem concerns the algebraic F{\o}lner conditions defined here, not topological amenability of a group action. Within this compact global-action setting, the weak and source-anchored properties coincide and depend only on the acting group. This does not assert the same equivalence for arbitrary ample groupoids, partial actions or groupoids of germs.
\end{remark}

\subsection{Finite-index induction}

Fix a subgroup $H\leq G$ of finite index $n$, and choose a set $T\subseteq G$ of representatives for the left cosets, with $e_G\in T$. Thus
$$
G=\bigsqcup_{t\in T}tH,\qquad |T|=n.
$$
Set $A=A_G$ and $B=A_H$. No normality assumption on $H$ will be used.

\begin{lemma}\label{finite-index-coset-model}
There is a direct-sum decomposition of right $B$-modules
$$
A=\bigoplus_{t\in T}u_tB.
$$
For each $g\in G$ and $t\in T$, let $\sigma_g(t)\in T$ and $h(g,t)\in H$ be determined by
$$
gt=\sigma_g(t)h(g,t).
$$
Then $\sigma_g$ permutes $T$. If $D_0\subseteq X$ is clopen, then
$$
\mathbf 1_{D_0\times\{g\}}u_t=u_{\sigma_g(t)}\bigl(\mathbf 1_{\sigma_g(t)^{-1}D_0}u_{h(g,t)}\bigr),
$$
where the expression in parentheses belongs to $B$.
\end{lemma}

\begin{proof}
Every term $u_gd$, with $d\in D$, can be written uniquely as $u_t(u_hd)$ once $g=th$ is its coset decomposition. Hence the displayed spaces span $A$. The support of $u_tB$ has labels in $tH$, so different summands have disjoint sets of labels and their sum is direct. Left multiplication by $u_t$ is a linear isomorphism from $B$ onto $u_tB$, with inverse left multiplication by $u_{t^{-1}}$.

Left multiplication by $g$ is a bijection of the finite coset space $G/H$, giving the permutation $\sigma_g$. Finally, Lemma~\ref{transformation-homogeneous-model} yields
$$
\begin{aligned}
u_{\sigma_g(t)}\bigl(\mathbf 1_{\sigma_g(t)^{-1}D_0}u_{h(g,t)}\bigr) &=\mathbf 1_{D_0}u_{\sigma_g(t)h(g,t)}\\
&=\mathbf 1_{D_0}u_{gt}=\mathbf 1_{D_0\times\{g\}}u_t.
\end{aligned}
$$
Since $h(g,t)\in H$, the coefficient belongs to $B$.
\end{proof}

\begin{proposition}[Quantitative induction]
\label{finite-index-folner-induction}
Let $L\subseteq A$ be finite-dimensional. There is a finite-dimensional subspace $E\subseteq B$, containing $1_B$, with the following property. For every non-empty clopen $C\subseteq X$ and every non-zero finite-dimensional $W_0\subseteq B\mathbf 1_C$, the space
$$
W=\bigoplus_{t\in T}u_tW_0
$$
satisfies
$$
W\subseteq A\mathbf 1_C,\ \dim_{\Bbbk}(W)=n\dim_{\Bbbk}(W_0),
$$
and
$$
\frac{\dim_{\Bbbk}(W+LW)-\dim_{\Bbbk}(W)}{\dim_{\Bbbk}(W)}\leq\frac{\dim_{\Bbbk}(W_0+EW_0)-\dim_{\Bbbk}(W_0)}{\dim_{\Bbbk}(W_0)}.
$$
\end{proposition}

\begin{proof}
Choose a basis $a_1,\dots,a_q$ of $L$. For each $i$ and $t\in T$, use Lemma~\ref{finite-index-coset-model} to write
$$
a_i u_t=\sum_{s\in T}u_s b_{i,s,t},\ b_{i,s,t}\in B.
$$
Set
$$
E=\operatorname{span}_{\Bbbk}\bigl(\{1_B\}\cup\{b_{i,s,t}\mid 1\leq i\leq q,\ s,t\in T\}\bigr).
$$
This space is finite-dimensional, depends only on $L$ and the chosen transversal, and works also for $L=\{0\}$ by taking $E=\Bbbk1_B$.

For $w\in W_0$, we have $w\mathbf 1_C=w$, so
$$
(u_tw)\mathbf 1_C=u_t(w\mathbf 1_C)=u_tw.
$$
Thus every $u_tW_0$ lies in $A\mathbf 1_C$. Their sum is direct by the coset support decomposition, and each has dimension $\dim_{\Bbbk}(W_0)$ because $u_t$ is invertible. This proves the containment and dimension assertion.

Let $Q=W_0+EW_0$. Since $B\mathbf 1_C$ is a left ideal of $B$, we have $Q\subseteq B\mathbf 1_C$. For each $i,t$, and $w\in W_0$,
$$
a_i(u_tw)=\sum_{s\in T}u_s b_{i,s,t}w\in\bigoplus_{s\in T}u_sQ.
$$
The same direct sum contains $W$, because $W_0\subseteq Q$. Consequently,
$$
W+LW\subseteq\bigoplus_{s\in T}u_sQ,\ \dim_{\Bbbk}(W+LW)\leq n\dim_{\Bbbk}(Q).
$$
Subtract $\dim_{\Bbbk}(W)=n\dim_{\Bbbk}(W_0)$ and divide by this positive number. The required relative-boundary estimate follows.
\end{proof}

\begin{theorem}\label{finite-index-transformation}
Let $H\leq G$ have finite index. If $X\rtimes H$ is source-anchored left algebraically amenable over $\Bbbk$, then $X\rtimes G$ is source-anchored left algebraically amenable over $\Bbbk$.
\end{theorem}

\begin{proof}
Let $C\subseteq X$ be non-empty and clopen, let $L\subseteq A_G$ be finite-dimensional, and let $\varepsilon>0$. Choose $E\subseteq A_H$ from Proposition~\ref{finite-index-folner-induction}. By the hypothesis and Proposition~\ref{anchored-finite-dimensional-test}, there exists a non-zero finite-dimensional subspace $W_0\subseteq A_H\mathbf 1_C$ such that
$$
\dim_{\Bbbk}(W_0+EW_0)\leq(1+\varepsilon)\dim_{\Bbbk}(W_0).
$$
Set $W=\bigoplus_{t\in T}u_tW_0$. Quantitative induction gives $W\subseteq A_G\mathbf 1_C$, $W\neq\{0\}$, and
$$
\dim_{\Bbbk}(W+LW)\leq(1+\varepsilon)\dim_{\Bbbk}(W).
$$
Proposition~\ref{anchored-finite-dimensional-test} now proves the assertion.
\end{proof}

\begin{remark}
The induction argument keeps the same source anchor $C$. It uses neither normality of $H$ nor division by $[G:H]$ inside $\Bbbk$, and therefore applies in every characteristic. Its conclusion could also be deduced from Theorem~\ref{compact-transformation-criterion} and the classical finite-index permanence property for groups. The direct proof supplies additional quantitative information: the dimension is multiplied by $[G:H]$, but the relative boundary does not increase.
\end{remark}

\subsection{Restrictions, finite orbits and examples}

\begin{corollary}\label{compact-action-subgroup-restriction}
Let $H\leq G$ be any subgroup. If $X\rtimes G$ is algebraically amenable over $\Bbbk$, then $X\rtimes H$ is algebraically amenable over $\Bbbk$. If $[G:H]<\infty$, then
$$
X\rtimes H\text{ is algebraically amenable}\quad\Longleftrightarrow\quad X\rtimes G\text{ is algebraically amenable}.
$$
\end{corollary}

\begin{proof}
By Theorem~\ref{compact-transformation-criterion}, amenability of $X\rtimes G$ implies amenability of $G$. Its subgroup $H$ is amenable by the classical subgroup property recalled in Section~\ref{preliminaries}. Proposition~\ref{amenable-compact-transformation}, applied to the restricted action, gives amenability of $X\rtimes H$. When the index is finite, the converse is Theorem~\ref{finite-index-transformation}.
\end{proof}

This corollary concerns subgroupoids obtained by restricting a fixed compact global action to a subgroup. It does not assert inheritance by arbitrary open subgroupoids, which fails by Proposition~\ref{arbitrary-subgroupoids-fail}.

\begin{corollary}\label{compact-action-finite-orbit}
Suppose that $O\subseteq X$ is a finite clopen $G$-orbit and let $x\in O$. Then the following conditions are equivalent:
\begin{enumerate}
\item $X\rtimes G$ is algebraically amenable over $\Bbbk$;
\item $(X\rtimes G)|_O$ is algebraically amenable over $\Bbbk$;
\item the stabilizer $G_x=\{g\in G\mid gx=x\}$ is amenable;
\item $G$ is amenable.
\end{enumerate}
\end{corollary}

\begin{proof}
The reduction is $O\rtimes G$. It is a countable transitive discrete groupoid because $O$ is a finite Hausdorff space. Its isotropy at $x$ consists of $(x,g)$ with $g\in G_x$, and is isomorphic to $G_x$. Thus Theorem~\ref{finite-transitive-isotropy} gives $(2)\Leftrightarrow(3)$. Applying Theorem~\ref{compact-transformation-criterion} first to $X$ and then to the non-empty compact space $O$ gives $(1)\Leftrightarrow(4)\Leftrightarrow(2)$. The clopen hypothesis also makes the passage from $(1)$ to $(2)$ an instance of Proposition~\ref{open-invariant-reduction}.
\end{proof}

\begin{example}[A dihedral action on a Cantor space]
\label{dihedral-compact-action}
Let $X=\{0,1\}^{\mathbb Z}$, and define homeomorphisms $T,J:X\to X$ by
$$
(Tx)_j=x_{j+1},\qquad (Jx)_j=x_{-j}.
$$
They satisfy $J^2=\operatorname{id}_X$ and $JTJ=T^{-1}$. Hence $r\mapsto T$ and $s\mapsto J$ define an action of the infinite dihedral group
$$
D_{\infty}=\langle r,s\mid s^2=e,\ srs=r^{-1}\rangle.
$$
The subgroup $H=\langle r\rangle\cong\mathbb Z$ has index two and transversal $\{e,s\}$. Proposition~\ref{amenable-compact-transformation} makes $X\rtimes H$ algebraically amenable, and Theorem~\ref{finite-index-transformation} gives the same conclusion for $X\rtimes D_{\infty}$.

More explicitly, for a prescribed anchor $C$, multiplier space $L$ and error $\varepsilon$, form the coefficient space $E$ from Proposition~\ref{finite-index-folner-induction}. Choose an anchored $(E,\varepsilon)$-F{\o}lner subspace $W_0$ for $A_H$ by the clopen-partition construction above. Then
$$
W=W_0\oplus u_sW_0,\ \dim_{\Bbbk}(W)=2\dim_{\Bbbk}(W_0),
$$
is an anchored $(L,\varepsilon)$-F{\o}lner subspace for $A_{\Bbbk}(X\rtimes D_{\infty})$. This gives a concrete noncommutative example on a nondiscrete compact unit space, with no restriction on the characteristic of $\Bbbk$.
\end{example}

\begin{example}
Let $X=\{0,1\}^{\mathbb N}$, let $G=F_2$ act trivially on $X$, and take $H=\{e_G\}$. Then $X\rtimes H$ is the unit groupoid on $X$ and is algebraically amenable. However, $X\rtimes G$ is not algebraically amenable by Theorem~\ref{compact-transformation-criterion}, since $F_2$ is not amenable. Here $[G:H]=\infty$, so this does not contradict the induction theorem. Thus amenability of a restricted action alone cannot justify ascent to an arbitrary larger acting group.
\end{example}

\begin{example}
\label{noncompact-regular-action}
Let $G=F_2$ and let $X=G$ with the discrete topology and the left translation action. This space is locally compact and zero-dimensional, but not compact. The map
$$
X\rtimes G\longrightarrow G\times G,\ (x,g)\longmapsto(x,g^{-1}x),
$$
is an isomorphism onto the pair groupoid, with inverse $(x,y)\mapsto(x,xy^{-1})$. It intertwines range, source and composition, and both topologies are discrete. The pair groupoid is transitive with infinite unit space, so Lemma~\ref{infinite-transitive-discrete-groupoid} shows that $X\rtimes G$ is source-anchored algebraically amenable even though $G$ is not amenable.

In this example, $X\times\{g\}$ is not compact, so its characteristic function is not an element of the Steinberg algebra. The global units $u_g$ used in the compact-action criterion are therefore absent. This explains why that argument does not rule out range escape for noncompact actions.
\end{example}

\section{Examples}\label{examples}

We now apply the preceding results to concrete Hausdorff ample groupoids. The examples include explicit F{\o}lner subspaces, finite-dimensional approximations, and obstructions that distinguish weak, source-anchored and hereditary algebraic amenability. Throughout this section, $\Bbbk$ is an arbitrary field, and algebraic amenability of a groupoid means source-anchored left algebraic amenability unless stated otherwise.

By Proposition~\ref{anchored-finite-dimensional-test}, it suffices in each positive example to start with a non-empty compact open subset $C$ of the unit space and a finite-dimensional multiplier space $L$, and construct non-zero finite-dimensional subspaces $W\subseteq A_{\Bbbk}(\mathcal G)\mathbf 1_C$ satisfying
$$
\frac{\dim_{\Bbbk}(W+LW)-\dim_{\Bbbk}(W)}{\dim_{\Bbbk}(W)}\leq\varepsilon.
$$
The constructions below show several distinct ways in which this boundary can be made small.

\subsection{Explicit F{\o}lner subspaces}

\begin{example}[Unit groupoids]\label{ex:unit-groupoids}
Let $X$ be a non-empty locally compact Hausdorff space with a basis of compact open subsets. Regard $X$ as a groupoid $\mathcal U_X$ containing only unit arrows. Thus $r(x)=s(x)=x$, and the only composable pairs are $(x,x)$, with product $xx=x$. Its Steinberg algebra is
$$
A_{\Bbbk}(\mathcal U_X)=C_c^{\mathrm{lc}}(X,\Bbbk),
$$
where $C_c^{\mathrm{lc}}(X,\Bbbk)$ denotes the algebra of locally constant, compactly supported functions, with pointwise multiplication. This follows directly from the convolution construction recalled in Section~\ref{preliminaries}; see also \cite[\S 3]{Ste09}.

Fix a non-empty compact open subset $C\subseteq X$ and a finite-dimensional subspace $L$ of this algebra. Choose a basis $f_1,\dots,f_q$ of $L$ and a point $x_0\in C$. Local constancy and the compact open basis give a non-empty compact open neighborhood $D$ of $x_0$, contained in $C$, on which every $f_i$ is constant. When $L=\{0\}$, no local constancy conditions are needed, and one may take $D=C$.

Set
$$
W=\operatorname{span}_{\Bbbk}\{\mathbf 1_D\}.
$$
Then $W\subseteq A_{\Bbbk}(\mathcal U_X)\mathbf 1_C$ and $\dim_{\Bbbk}(W)=1$. For every $i$,
$$
f_i\mathbf 1_D=f_i(x_0)\mathbf 1_D,
$$
so $LW\subseteq W$. Hence $W+LW=W$, and the algebraic boundary is zero. Therefore $\mathcal U_X$ is algebraically amenable over every field.

For instance, taking $X=\{0,1\}^{\mathbb N}$ gives a nondiscrete example with an infinite-dimensional algebra: the characteristic functions of the binary cylinders of any fixed length are linearly independent, and the number of such cylinders is unbounded. Moreover, every Hausdorff ample subgroupoid of $\mathcal U_X$ is itself a unit groupoid, so the same argument proves hereditary algebraic amenability in this family.
\end{example}

\begin{example}[Constant free abelian group bundles]
\label{ex:abelian-group-bundles}
Let $X$ be as in Example~\ref{ex:unit-groupoids}, and let $d\geq 1$. Define
$$
\mathcal B_d=X\times\mathbb Z^d,\ r(x,g)=s(x,g)=x,
$$
with multiplication $(x,g)(x,h)=(x,g+h)$ and inversion $(x,g)^{-1}=(x,-g)$. Equip $\mathcal B_d$ with the product topology, using the discrete topology on $\mathbb Z^d$. The sets $D\times\{g\}$, with $D\subseteq X$ compact open, form a basis of compact open bisections. Thus $\mathcal B_d$ is Hausdorff and ample. All its orbits are singletons, while its isotropy group at every unit is $\mathbb Z^d$.

For $f\in C_c^{\mathrm{lc}}(X,\Bbbk)$ and $g\in\mathbb Z^d$, write $f\delta_g$ for the function defined by
$$
(f\delta_g)(x,h)=f(x)\mathbf 1_{\{g\}}(h).
$$
Compact support implies that every element of $A_{\Bbbk}(\mathcal B_d)$ is a finite sum of such functions. Convolution gives
$$
(f\delta_g)\ast(h\delta_t)=(fh)\delta_{g+t}.
$$
Consequently,
$$
A_{\Bbbk}(\mathcal B_d)\cong C_c^{\mathrm{lc}}(X,\Bbbk)\otimes_{\Bbbk}\Bbbk\mathbb Z^d.
$$

We give a direct source-anchored F{\o}lner construction. Fix a non-empty compact open subset $C\subseteq X$, a finite-dimensional subspace $L\subseteq A_{\Bbbk}(\mathcal B_d)$, and $\varepsilon>0$. Choose a basis $a_1,\dots,a_q$ of $L$ and a finite set $T\subseteq\mathbb Z^d$ such that
$$
a_i=\sum_{g\in T}f_{i,g}\delta_g,\ f_{i,g}\in C_c^{\mathrm{lc}}(X,\Bbbk).
$$
Choose $x_0\in C$ and a non-empty compact open neighborhood $D\subseteq C$ of $x_0$ on which every coefficient function $f_{i,g}$ is constant. The case $L=\{0\}$ is included by taking $T=\varnothing$ and $D=C$.

Choose an integer $m\geq 0$ such that $T\cup\{0\}\subseteq[-m,m]^d$, and for every integer $N\geq 1$ put
$$
F_N=[-N,N]^d\cap\mathbb Z^d,\ W_N=\operatorname{span}_{\Bbbk}\{\mathbf 1_D\delta_f\mid f\in F_N\}.
$$
The basis functions have disjoint supports, so
$$
W_N\subseteq A_{\Bbbk}(\mathcal B_d)\mathbf 1_C,\ \dim_{\Bbbk}(W_N)=(2N+1)^d.
$$
For each $i$ and $f\in F_N$,
$$
a_i\ast(\mathbf 1_D\delta_f)=\sum_{g\in T}f_{i,g}(x_0)\mathbf 1_D\delta_{g+f}.
$$
Since $g+f\in F_{N+m}$, this implies $W_N+LW_N\subseteq W_{N+m}$. Therefore
$$
\frac{\dim_{\Bbbk}(W_N+LW_N)}{\dim_{\Bbbk}(W_N)}\leq\left(\frac{2N+2m+1}{2N+1}\right)^d\longrightarrow 1.
$$
For sufficiently large $N$, the required F{\o}lner inequality follows. Thus $\mathcal B_d$ is algebraically amenable over every field.

When $X$ is a point, the algebra is the Laurent polynomial algebra $\Bbbk[t_1^{\pm1},\dots,t_d^{\pm1}]$. For a Cantor space $X$, the same construction gives nondiscrete examples with nontrivial infinite isotropy. In this case the approximation takes place in the isotropy directions, rather than by escaping to new range units.
\end{example}

\begin{example}[Finite matrix groupoids]
\label{ex:finite-matrix-groupoids}
Let $I_n=\{1,\dots,n\}$, where $n\geq 1$, and let $\Gamma$ be a countable discrete group. The discrete transitive groupoid
$$
\mathcal M_n(\Gamma)=I_n\times\Gamma\times I_n
$$
has range $r(i,g,j)=i$, source $s(i,g,j)=j$, multiplication
$$
(i,g,j)(j,h,k)=(i,gh,k),
$$
and inversion $(i,g,j)^{-1}=(j,g^{-1},i)$. Its matrix model is
$$
A_{\Bbbk}(\mathcal M_n(\Gamma))\cong M_n(\Bbbk\Gamma),\ \mathbf 1_{\{(i,g,j)\}}\longmapsto E_{ij}(g),
$$
where $E_{ij}(a)$ denotes the matrix with entry $a$ in position $(i,j)$ and zero elsewhere. This is the model of Lemma~\ref{finite-transitive-matrix-model}; see also \cite[Example 3.2 and Proposition 4.3]{ClarkHazrat20}.

For $\Gamma=\mathbb Z$, identify $\Bbbk\Gamma$ with $\Bbbk[t,t^{-1}]$. Let $C\subseteq I_n$ be non-empty, choose $j\in C$, and write
$$
e_C=\sum_{c\in C}E_{cc}(1).
$$
Given a finite-dimensional subspace $L\subseteq M_n(\Bbbk[t,t^{-1}])$, choose $m\geq 0$ such that every entry of every matrix in $L$ is a linear combination of $t^p$ with $|p|\leq m$. For $N\geq 1$, set
$$
W_N=\operatorname{span}_{\Bbbk}\{E_{ij}(t^q)\mid 1\leq i\leq n,\ -N\leq q\leq N\}.
$$
Then $W_N\subseteq M_n(\Bbbk[t,t^{-1}])e_C$ and
$$
\dim_{\Bbbk}(W_N)=n(2N+1).
$$
The identity
$$
E_{ab}(t^p)E_{ij}(t^q)=\delta_{b,i}E_{aj}(t^{p+q})
$$
shows that $W_N+LW_N\subseteq W_{N+m}$. Hence
$$
\frac{\dim_{\Bbbk}(W_N+LW_N)}{\dim_{\Bbbk}(W_N)}\leq\frac{2N+2m+1}{2N+1}=1+\frac{2m}{2N+1}.
$$
Thus finite matrix amplification preserves these explicit F{\o}lner approximations.

More generally, Theorem~\ref{finite-transitive-isotropy} gives
$$
\mathcal M_n(\Gamma)\text{ is algebraically amenable}\quad\Longleftrightarrow\quad
\Gamma\text{ is amenable}.
$$
In particular, $\mathcal M_n(\mathbb Z^d)$ is algebraically amenable, whereas $\mathcal M_n(F_2)$ is not, since $F_2$ is non-amenable; see \cite[Theorem 5]{Cortez22}. Thus neither a finite unit space nor finite matrix amplification removes the obstruction carried by non-amenable isotropy.
\end{example}

\subsection{An AF groupoid with Cantor unit space}

\begin{example}[Binary tail equivalence]
\label{ex:binary-tail-equivalence}
Let $X=\{0,1\}^{\mathbb N}$ with the product topology, and let $\sigma:X\to X$ be the left shift. Consider the tail-equivalence groupoid
$$
\mathcal R=\{(x,y)\in X\times X\mid\sigma^n(x)=\sigma^n(y)\text{ for some }n\geq 0\},
$$
with $r(x,y)=x$, $s(x,y)=y$, inversion $(x,y)^{-1}=(y,x)$, and multiplication $(x,y)(y,z)=(x,z)$.

If $\alpha$ and $\beta$ are binary words of the same finite length, define
$$
Z(\alpha,\beta)=\{(\alpha z,\beta z)\mid z\in X\}.
$$
Here $\alpha z$ denotes concatenation. Equip $\mathcal R$ with the topology having these sets as a basis, rather than the subspace topology inherited from $X\times X$. Each $Z(\alpha,\beta)$ is a compact open bisection, with range $\alpha X$ and source $\beta X$. These charts make $\mathcal R$ a Hausdorff ample groupoid with unit space $X$.

For $n\geq 0$, put
$$
\mathcal R_n=\{(x,y)\in X\times X\mid\sigma^n(x)=\sigma^n(y)\}.
$$
It is the finite disjoint union of the bisections $Z(\alpha,\beta)$ with $|\alpha|=|\beta|=n$. Hence $\mathcal R_n$ is compact, open and principal, and $\mathcal R_n^{(0)}=X$. Moreover,
$$
\mathcal R_n\subseteq\mathcal R_{n+1},\ \mathcal R=\bigcup_{n\geq 0}\mathcal R_n.
$$
Thus $\mathcal R$ is an AF groupoid in the sense of Definition~\ref{af-groupoid-definition}; compare \cite[\S 5.2]{ClarkHazrat20}.

To exhibit the F{\o}lner subspaces explicitly, set
$$
b_{\alpha,\beta}=\mathbf 1_{Z(\alpha,\beta)}, \ B_n=\operatorname{span}_{\Bbbk}
\{b_{\alpha,\beta}\mid |\alpha|=|\beta|=n\}.
$$
For words of length $n$, convolution gives
$$
b_{\alpha,\beta}\ast b_{\gamma,\delta}=\delta_{\beta,\gamma}b_{\alpha,\delta}.
$$
The supports are pairwise disjoint, so these functions are linearly independent matrix units. Consequently,
$$
B_n\cong M_{2^n}(\Bbbk),\ \sum_{|\alpha|=n}b_{\alpha,\alpha}=\mathbf 1_X.
$$
The disjoint decomposition
$$
Z(\alpha,\beta)=Z(\alpha0,\beta0)\sqcup Z(\alpha1,\beta1)
$$
yields
$$
b_{\alpha,\beta}=b_{\alpha0,\beta0}+b_{\alpha1,\beta1},
$$
so $B_n\subseteq B_{n+1}$. Every locally constant, compactly supported function is a finite linear combination of characteristic functions of these basic bisections: cover its compact support by finitely many such bisections on which it is constant, and refine them to a common word length. Therefore
$$
A_{\Bbbk}(\mathcal R)=\bigcup_{n\geq 0}B_n.
$$
The finite-dimensional algebras $B_n$ should not be confused with the full algebras $A_{\Bbbk}(\mathcal R_n)$, which allow arbitrary locally constant dependence on the common infinite tail.

Now let $C\subseteq X$ be non-empty compact open and let $L\subseteq A_{\Bbbk}(\mathcal R)$ be finite-dimensional. Choose $n$ large enough that $L\subseteq B_n$ and $C$ is a union of cylinders of length $n$. If $\mathcal C\subseteq\{0,1\}^n$ is the non-empty collection of words defining those cylinders, then
$$
\mathbf 1_C=\sum_{\beta\in\mathcal C}b_{\beta,\beta}\in B_n.
$$
Set $W=B_n\mathbf 1_C$. The matrix units give
$$
W=\operatorname{span}_{\Bbbk}\{b_{\alpha,\beta}\mid |\alpha|=n,\ \beta\in\mathcal C\},\ \dim_{\Bbbk}(W)=2^n|\mathcal C|.
$$
In particular, $W\neq\{0\}$ and $W\subseteq A_{\Bbbk}(\mathcal R)\mathbf 1_C$. Since $L\subseteq B_n$,
$$
LW\subseteq B_n(B_n\mathbf 1_C)\subseteq B_n\mathbf 1_C=W.
$$
Thus the algebraic boundary is zero. This realizes Corollary~\ref{af-groupoids} concretely in an infinite-dimensional, noncommutative Steinberg algebra with nondiscrete unit space.
\end{example}

\subsection{The role of anchoring and orbit size}

\begin{example}[Infinite matrix amplification]\label{ex:infinite-matrix-amplification}
Let $\Gamma=F_2$, and equip
$$
\mathcal M_{\infty}(\Gamma)=\mathbb N\times\Gamma\times\mathbb N
$$
with the discrete topology and the same operations as in Example~\ref{ex:finite-matrix-groupoids}. Its Steinberg algebra identifies with
$$
M_{\mathbb N}^{\mathrm{fin}}(\Bbbk\Gamma),
$$
the algebra of matrices over $\Bbbk\Gamma$ with only finitely many non-zero entries in total. Indeed, the assignment $\mathbf 1_{\{(i,g,j)\}}\mapsto E_{ij}(g)$ maps the natural basis to the matrix group-element basis and respects multiplication.

Let $C\subseteq\mathbb N$ be non-empty and finite, and let $L$ be a finite-dimensional multiplier space. The union $S$ of the supports of a basis of $L$ is finite. Consequently,
$$
s(S)=\{s(\alpha)\mid\alpha\in S\}
$$
is finite. Choose $i_0\in\mathbb N\setminus s(S)$ and $c_0\in C$, and put
$$
\beta=(i_0,e_{\Gamma},c_0),\ W=\operatorname{span}_{\Bbbk}\{\mathbf 1_{\{\beta\}}\}.
$$
Since $s(\beta)=c_0$, the space $W$ is contained in $A_{\Bbbk}(\mathcal M_{\infty}(\Gamma))\mathbf 1_C$. No $\alpha\in S$ can multiply $\beta$ on the left, because composability would require $s(\alpha)=i_0$. Hence
$$
LW=\{0\},\ \dim_{\Bbbk}(W+LW)=\dim_{\Bbbk}(W)=1.
$$
Thus $\mathcal M_{\infty}(F_2)$ is algebraically amenable over every field, as predicted by Lemma~\ref{infinite-transitive-discrete-groupoid}.

Nevertheless, for every $u\in\mathbb N$,
$$
\mathcal M_{\infty}(F_2)|_{\{u\}}=\{u\}\times F_2\times\{u\}\cong F_2
$$
is not algebraically amenable by Theorem~\ref{group-case}. In particular, the ambient groupoid is not hereditarily algebraically amenable. The set $\{u\}$ is compact open but not invariant, so this does not contradict Proposition~\ref{open-invariant-reduction}.

There is also a contrast with the directed-union result. Writing
$$
\mathcal M_{\infty}(F_2)=\bigcup_{n\geq 1}\mathcal M_n(F_2),
$$
we obtain an algebraically amenable directed union whose every displayed finite-unit stage is non-amenable. Thus the implication in Proposition~\ref{directed-unions} cannot be reversed. The zero boundary here is caused by range escape, not by amenability of the isotropy groups.
\end{example}

\begin{example}[An amenable summand does not suffice for anchoring]
\label{ex:weak-versus-anchored}
Let $\mathcal T$ be the groupoid consisting of one unit and consider
$$
\mathcal D=\mathcal T\sqcup F_2.
$$
Then
$$
A_{\Bbbk}(\mathcal D)\cong\Bbbk\oplus\Bbbk F_2.
$$
For every finite-dimensional subspace $L$ of this algebra, the fixed one-dimensional space
$$
W=\Bbbk\oplus\{0\}
$$
satisfies $LW\subseteq W$, because multiplication is componentwise. Therefore the algebra is algebraically amenable, and $\mathcal D$ is weakly bisectionally algebraically amenable by Proposition~\ref{weak-bisectional-equivalence}.

Let $u$ be the unit of the $F_2$ component. Anchoring at $C=\{u\}$ gives
$$
A_{\Bbbk}(\mathcal D)\mathbf 1_C=\{0\}\oplus\Bbbk F_2.
$$
If $\mathcal D$ were source-anchored algebraically amenable, testing this anchor against every multiplier space of the form $\{0\}\oplus E$, with $E\subseteq\Bbbk F_2$ finite-dimensional, would make $\Bbbk F_2$ algebraically amenable. This contradicts Theorem~\ref{bartholdi-group-algebra}. Thus the weak and anchored properties differ even for a unital Steinberg algebra with a finite unit space.
\end{example}

\subsection{A nondiscrete obstruction with amenable isotropy}

\begin{example}[The two-loop graph groupoid]\label{ex:two-loop-obstruction}
Retain $X=\{0,1\}^{\mathbb N}$ and the shift $\sigma$ from Example~\ref{ex:binary-tail-equivalence}. Define
$$
\mathcal G_2=\{(x,m-n,y)\mid x,y\in X,\ m,n\geq 0,\ \sigma^m(x)=\sigma^n(y)\}.
$$
The range and source maps are $r(x,k,y)=x$ and $s(x,k,y)=y$, and the operations are
$$
(x,k,y)(y,\ell,z)=(x,k+\ell,z),\ (x,k,y)^{-1}=(y,-k,x).
$$
For arbitrary finite binary words $\alpha,\beta$, including the empty word, put
$$
Z_2(\alpha,\beta)=\{(\alpha z,|\alpha|-|\beta|,\beta z)\mid z\in X\}.
$$
These sets form the standard basis of compact open bisections. With this topology, $\mathcal G_2$ is the Hausdorff ample groupoid of the graph with one vertex and two loops; see \cite[\S 3.5]{ClarkHazrat20}.

Let $\varnothing$ denote the empty word, and, for $i\in\{0,1\}$, set
$$
U_i=Z_2(i,\varnothing),\ p_i=\mathbf 1_{U_i},\ q_i=\mathbf 1_{U_i^{-1}}.
$$
Each $U_i$ has source $X$ and range $iX$. The range sets $0X$ and $1X$ are disjoint and cover $X$. Bisection multiplication therefore gives
$$
q_i\ast p_j=\delta_{i,j}\mathbf 1_X,\ p_0\ast q_0+p_1\ast q_1=\mathbf 1_X.
$$
These are the familiar relations of the Leavitt algebra $L_{\Bbbk}(1,2)$, which identifies with $A_{\Bbbk}(\mathcal G_2)$; see \cite[Example 4.4]{ClarkHazrat20} and \cite[Definition 5.1]{AraLiLledoWu18}.

The relations give a direct obstruction to algebraic F{\o}lner subspaces. Write $A=A_{\Bbbk}(\mathcal G_2)$ and take any non-zero finite-dimensional subspace $W\subseteq A$. The linear map
$$
W\oplus W\longrightarrow A,\ (w_0,w_1)\longmapsto p_0\ast w_0+p_1\ast w_1,
$$
is injective: multiplying a zero image on the left by $q_0$ gives $w_0=0$, and multiplying by $q_1$ gives $w_1=0$. Consequently,
$$
\dim_{\Bbbk}(p_0W+p_1W)=2\dim_{\Bbbk}(W).
$$
For the fixed two-dimensional multiplier space
$$
L=\operatorname{span}_{\Bbbk}\{p_0,p_1\},
$$
we therefore have
$$
\dim_{\Bbbk}(W+LW)\geq 2\dim_{\Bbbk}(W)
$$
for every non-zero finite-dimensional $W\subseteq A$. Taking $0<\varepsilon<1$ excludes the F{\o}lner inequality. Thus $A$ is not algebraically amenable, and $\mathcal G_2$ is neither weakly nor source-anchored algebraically amenable over any field. This calculation is the concrete two-loop instance of the obstruction discussed in \cite[Corollary 4.9]{AraLiLledoWu18}.

Nevertheless, every isotropy group of $\mathcal G_2$ is amenable. Indeed, for each $x\in X$, the map
$$
(\mathcal G_2)_x^x\longrightarrow\mathbb Z,\ (x,k,x)\longmapsto k,
$$
is an injective group homomorphism. Hence the isotropy is trivial or infinite cyclic. Every orbit is also infinite: the distinct points $0^n1x$, for $n\geq 0$, all belong to the orbit of $x$, since $\sigma^{n+1}(0^n1x)=x$.

Thus amenable isotropy and infinite orbits do not imply source-anchored algebraic amenability in a nondiscrete groupoid. There is no contradiction with Theorem~\ref{hereditary-discrete-characterization} or Lemma~\ref{infinite-transitive-discrete-groupoid}, both of which require a discrete groupoid.

Finally, the map $(x,y)\mapsto(x,0,y)$ identifies the AF groupoid $\mathcal R$ of Example~\ref{ex:binary-tail-equivalence} with the open subgroupoid of $\mathcal G_2$ consisting of degree-zero arrows. These two groupoids have the same unit space, but only $\mathcal R$ is algebraically amenable. In $\mathcal G_2$, the compact open bisections $U_0,U_1$ have full source $X$, but their ranges are the proper clopen subsets $0X$ and $1X$. Their left multipliers act on every source anchor, so moving the range cannot avoid the uniform dimension growth.
\end{example}

\begin{remark}
The examples separate the three approximation properties considered in the paper. Unit groupoids provide hereditary examples; infinite matrix amplification is source-anchored algebraically amenable but not hereditary; and $\mathcal T\sqcup F_2$ satisfies only the weak bisectional algebraic condition. The two-loop groupoid gives a further obstruction that is not detected by its individual isotropy groups. Thus the conditions depend not only on the groups attached to units, but also on how compact open bisections act across the unit space.
\end{remark}

\section{Future work}\label{future-work}

The compact-action criterion in Section~\ref{finite-index-transformation-groupoids} identifies precisely what the present condition detects for global actions on non-empty compact zero-dimensional spaces: amenability of the acting group. A subsequent study should therefore concentrate on noncompact global actions and on partial actions, where the global bisections used in that proof need not belong to the Steinberg algebra. In these settings, one should determine how compact source anchors, stabilizers and the availability of arrows with distant ranges affect the F{\o}lner condition. The regular action in Example~\ref{noncompact-regular-action} gives a concrete reason why the compact criterion cannot simply be extended unchanged.

The quantitative induction estimate suggests a separate line of investigation. Rather than asking again about finite-index subgroups of compact global actions, one can ask which inclusions of ample subgroupoids admit a finite decomposition compatible with the source anchors and with multiplication by compact open bisections. Such a decomposition could extend the induction argument to suitable partial actions, cocycle constructions and finite extensions beyond transformation groupoids. Establishing precise hypotheses is essential, since openness of a subgroupoid alone does not imply inheritance of source-anchored algebraic amenability.

An algebraic alternative to the F{\o}lner condition should also account for more than one type of obstruction. Non-amenable isotropy in a one-unit groupoid and the two disjoint copies produced by the bisections of the two-loop groupoid are different mechanisms. Thus paradoxicality of compact open subsets of the unit space alone cannot describe every failure of algebraic amenability. A useful formulation may have to combine bisectional decompositions with dimension-theoretic data from the Steinberg algebra, in the spirit of \cite{AraLiLledoWu18}. The examples in Section~\ref{examples} provide tests for any proposed criterion.

Finally, the failure of permanence under compact open reductions exhibited by infinite matrix amplification motivates approximation conditions that control the range as well as the source. One may investigate whether suitably strengthened conditions have better behavior under full reductions, groupoid equivalence or Morita equivalence, in light of the relation established in \cite{ClarkSims15}, rather than expecting the present condition to have these properties without additional assumptions. Dependence on the coefficient field remains another question beyond the classes treated here. The group case, the hereditary countable discrete case and compact global transformation groupoids are field-independent; understanding what remains true for general ample groupoids would help distinguish algebraic effects from those imposed by the topology.

\section*{Statements and Declarations}

\noindent\textbf{Funding.}
No funding was received to assist with the preparation of this manuscript.

\medskip
\noindent\textbf{Competing interests.}
The authors have no relevant financial or non-financial interests to disclose.

\medskip
\noindent\textbf{Data availability.}
No datasets were generated or analyzed during the current study.

\end{document}